\documentclass[12pt,reqno]{article}

\usepackage[usenames]{color}
\usepackage{amssymb}
\usepackage{graphicx}
\usepackage{amscd}

\usepackage{amsthm}
\newtheorem{theorem}{Theorem}

\newtheorem{proposition}[theorem]{Proposition}

\newtheorem{conjecture}[theorem]{Conjecture}

\theoremstyle{definition}

\newtheorem{example}[theorem]{Example}

\usepackage[colorlinks=true,
linkcolor=webgreen, filecolor=webbrown,
citecolor=webgreen]{hyperref}

\definecolor{webgreen}{rgb}{0,.5,0}
\definecolor{webbrown}{rgb}{.6,0,0}

\usepackage{color}

\usepackage{float}

\usepackage{graphics,amsmath,amssymb}
\usepackage{amsfonts}
\usepackage{latexsym}
\usepackage{epsf}

\newcommand{\seqnum}[1]{\href{http://oeis.org/#1}{\underline{#1}}}

\begin{document}

\begin{center}
\vskip 1cm{\LARGE\bf Some observations on the Rueppel sequence} \vskip 1cm \large
Paul Barry\\
School of Science\\
Waterford Institute of Technology\\
Ireland\\
\href{mailto:pbarry@wit.ie}{\tt pbarry@wit.ie}
\end{center}
\vskip .2 in

\begin{abstract} Starting with a definition based on the Catalan numbers, we carry out an empirical study of the Rueppel sequence. We use the Hankel transform as the main technique. By means of this transform we find links to such sequences as the Jacobi sequence and the paper-folding sequence.
\end{abstract}

\section{Introduction}
In this note we carry out a largely empirical study of a variation on a classical sequence (the Rueppel sequence) whose Hankel transform leads to another sequence of fundamental importance. In further exploring the family of ideas thrown up by this, we encounter other Hankel transforms that mirror other sequences regarded as basic. The mirroring arises due to the fact that these Hankel transforms are signed versions of more classical sequences. We end this note by conjecturing a form for this signing, again in terms of classical sequences.

The form of the Hankel transform of sequences $a_n$ that we use is to consider the sequence $h_n=|a_{i+j}|_{0 \le i,j \le n}$ as the Hankel transform of $a_n$. Thus this is a sequence of Hankel determinants. A careful analysis of the types of Hankel determinants encountered in this note has been carried out by a number of authors \cite{Bacher, Cigler, PD}.

The main building block of the paper is the Rueppel sequence $r_n$, whose elements can be defined by
$$r_n = C_n \bmod 2,$$
where $C_n=\frac{1}{n+1}\binom{2n}{n}$ is the $n$-th Catalan number. A consequence of this is that methods developed for moment sequences (such as the Hankel transform) are appropriate. Thus we shall see that some generating functions we will encounter can be described by continued fractions of Jacobi type \cite{Wall}.

If a sequence $a_n$ has a generating function $g(x)$ that can be expressed as a continued
$$g(x)=\cfrac{1}{1-\alpha_0 x-
\cfrac{\beta_1 x^2}{1-\alpha_1 x-
\cfrac{\beta_2 x^2}{1-\alpha_2 x-\cdots}}},$$ then the Hankel transform of $a_n$ is given by
$$h_n=\prod_{k=0}^n b_k^{n-k},$$  independent of the $\alpha_n$. This does mean that many sequences may have the same Hankel transform.

We make ample use of the treasure trove that resides in the On-Line Encyclopedia of Integer Sequences (OEIS) \cite{SL1, SL2}. Indeed, as will be appreciated by the reader, this paper owes a great deal of its existence to this resource. Sequences where known will be referred to by their OEIS number.

In parts of this paper, we will make use of the language of Riordan arrays \cite{book, SGWW}. The motivation for this comes from the analogy between the Rueppel numbers and the Catalan numbers, as Riordan arrays provide a powerful framework for dealing with important combinatorial triangles defined by Catalan numbers. A backdrop to this is the context of moment sequences and orthogonal polynomials, which we do not pursue here, but which has been touched on elsewhere \cite{Cigler}.

The OEIS sequence \seqnum{A007318} is Pascal's triangle, also known as the binomial matrix $B=\left(\binom{n}{k}\right)$. As a Riordan array, this is defined by
$$\left(\frac{1}{1-x}, \frac{x}{1-x}\right),$$ where
$$\binom{n}{k}=[x^n]\frac{1}{1-x} \left(\frac{x}{1-x}\right)^k.$$ Here, $[x^n]$ is the functional which extracts the coefficient of $x^n$ when applied to a  power series $\sum_{n=0}^{\infty} a_n x^n$. In general a Riordan array is defined by a pair of suitable power series $(g(x), f(x))$. The fundamental theorem of Riordan arrays is says that
$$(g(x), f(x))\cdot A(x)=g(x)f(A(x)).$$

The basic sequences that we shall encounter are often related to binary representations of natural numbers. Significant sequences bear such names as the paper-folding sequence, the Jacobi (symbol) sequence, the Thue-Morse sequence, the Golay-Rudin-Shapiro sequence, Gould's sequence, Stern's diatomic sequence, and of course the Rueppel sequence and the Catalan numbers.

The plan of this paper is as follows.

\begin{itemize}
\item This introduction
\item Preliminaries on basic sequences
\item Distribution of elements in $s_n$
\item The Rueppel sequence
\item Modifying the Rueppel sequence and Hankel transforms
\item The sequence $\tilde{s}_n$
\item Some further results
\item Related explorations
\item Some notable number triangles and related sequences
\item The sequence $\sigma_n$
\item Conclusions
\end{itemize}

\section{Preliminaries on basic sequences}
We begin with a review of some important sequences which will be useful in the sequel.
The Jacobi sequence $j_n$ is the sequence with generating function
$$ j(x)=\sum_{k=0}^{\infty} \frac{x^{2^k}}{1+x^{2^{k+1}}},$$ which begins
$$0,1, 1, -1, 1, 1, -1, -1, 1, 1, 1, -1, -1, 1, -1, -1, 1,\ldots.$$ We have $j_0=0$, and
$$j_n=\left(\frac{-1}{n}\right) \quad \text{for}\quad n >0.$$
This is \seqnum{A034947}. The partial sums of this sequence, $g_n=\sum_{k=0}^n j_k$, have generating function
 $$ \frac{1}{1-x} \sum_{k=0}^{\infty} \frac{x^{2^k}}{1+x^{2^{k+1}}}.$$
This sequence begins
$$0, 1, 2, 1, 2, 3, 2, 1, 2, 3, 4, 3, 2, 3,\ldots.$$ This is \seqnum{A005811}, and it counts the number of runs in the binary expansion of $n > 0$, or alternatively, the number of $1$'s in the Gray code for $n$.

Of importance for this note is the sequence $$s_n=1+\sum_{k=0}^n j_k$$ obtained by adding $1$ to each element of this sequence, to get the sequence that begins
$$1, 2, 3, 2, 3, 4, 3, 2, 3, 4, 5, 4, 3, 4, 3, 2, 3, 4, \ldots.$$ This is \seqnum{A088748}. The generating function for $s_n$ is given by
$$\frac{1}{1-x}+\frac{1}{1-x} \sum_{k=0}^{\infty} \frac{x^{2^k}}{1+x^{2^{k+1}}}=\frac{1}{1-x}\left(1+ \sum_{k=0}^{\infty} \frac{x^{2^k}}{1+x^{2^{k+1}}}\right),$$ and so we have
$$s_n=\sum_{k=0}^n 0^k+j_k.$$ That is, $s_n$ represents the partial sums of the Jacobi sequence when this sequence starts $1,1,1,-1,1,\ldots$.

The sequence with general element $$p_n=\frac{j_{n+1}+1}{2}$$ is the \emph{paper-folding sequence} which begins
$$1, 1, 0, 1, 1, 0, 0, 1, 1, 1, 0, 0, 1, 0, 0, 1, 1, 1, 0,\ldots.$$ This is \seqnum{A014577}. It has its generating function given by
$$\frac{1}{2}\left(\frac{1}{1-x}+\frac{1}{x}\sum_{k=0}^{\infty} \frac{x^{2^k}}{1+x^{2^{k+1}}}\right)=\sum_{k=0}^{\infty} \frac{x^{2^k-1}}{1-x^{2^{k+2}}}.$$
We have the following relation between $s_n$ and $p_n$.
$$\frac{1+s_n-s_{n+1}}{2}=1-p_n,$$ or equivalently
$$\frac{s_{n+1}-s_n+1}{2}=p_n.$$
The sequence $\bar{p}_n=1-p_n$ is \seqnum{A014707}, where
$$\bar{p}_n =\frac{1-\left(\frac{-1}{n+1}\right)}{2}$$ satisfies
$\bar{p}_{4n}=0, \bar{p}_{4n+2}=1$, and $\bar{p}_{2n+1}=\bar{p}_n$.

Apart from the above definition of the paper-folding sequence, there are many links between the sequences $p_n$ and $j_n$. If we consider the positions of $+1$ in the Jacobi sequence $j_n$, we find that they occur at the positions indexed by the sequence \seqnum{A091072}, beginning
$$1, 2, 4, 5, 8, 9, 10, 13, 16, 17, 18, 20, 21, 25, 26, 29,\ldots.$$  The characteristic function of this sequence is the paper-folding sequence. Thus if we set all terms $-1$ of the sequence $\left(\frac{-1}{n}\right)$ to zero, we obtain the paper-folding sequence (with a $1$ pre-pended). Regarding the positions of the $-1$'s in the Jacobi sequence, we find that they are indexed by the complementary sequence to \seqnum{A091072}, namely \seqnum{A091067}. This sequence is composed of numbers $a_n$ such that the odd-part of $a_n$ is of the form $4k+3$. It begins $$3, 6, 7, 11, 12, 14, 15, 19, 22, 23, 24,\ldots.$$ The characteristic sequence for this is the sequence \seqnum{A038189}. This latter sequence begins
$$0, 0, 0, 1, 0, 0, 1, 1, 0, 0, 0, 1, 1, 0, 1, 1, 0, 0, 0, 1,\ldots.$$ The sequence that begins
$$ 0, 0, 1, 0, 0, 1, 1, 0, 0, 0, 1, 1, 0, 1, 1, 0, 0, 0, 1,\ldots$$ is \seqnum{A014707}, which can be specified by $t_0=0, t_{4n}=0, t_{4n+2}=1$, and $t_{2n+1}=t_n$.

\section{Distribution of elements in $s_n$}
We recall that the sequence $s_n$ is given by
$$s_n=1+\sum_{k=0}^n j_k$$ and begins
$$1, 2, 3, 2, 3, 4, 3, 2, 3, 4, 5, 4, 3, 4, 3, 2, 3, 4, \ldots.$$
It is interesting to look at the positions of the occurrences of the digits $1,2,3,\ldots$ in this sequence.
We display below a table representing the locations of the occurrences of these digits.
\begin{center}
\begin{tabular}{|c||c|c|c|c|c|c|}
\hline $n$ & $1^{\text{st}}$  & $2^{\text{nd}}$ & $3^{\text{rd}}$ & $4^{\text{th}}$ & Axxxxxx & Comment\\
\hline\hline $2$ & $1$ & $3$ & $7$ & $15$ & \seqnum{A000225}& $1$ run \\
\hline $3$ & $2$ & $4$ & $6$ & $8$ & \seqnum{A043569}  & $2$ runs\\
\hline $4$ & $5$ & $9$ & $11$ & $13$ & \seqnum{A043570} & $3$ runs\\
\hline $5$ & $10$ & $18$ & $20$ & $22$ & \seqnum{A043571} & $4$ runs\\
\hline $6$ & $21$ & $37$ & $41$ & $43$ & \seqnum{A043572} & $5$ runs \\
\hline $7$ & $42$ & $74$ & $82$ & $84$ & \seqnum{A043573} & $6$ runs\\
\hline $8$ & $85$ & $149$ & $165$ & $169$ & \seqnum{A043574} & $7$ runs\\
\hline $9$ & $170$ & $298$ & $330$ & $338$ & \seqnum{A043575} & $8$ runs\\
\hline $10$ & $341$ & & & & & $9$ runs\\
\hline
\end{tabular}
\end{center}
The ``first-occurrence'' sequence $0,1,2,5,10,21,\ldots$ is the Jacobsthal related sequence \seqnum{A000975}, with general term
$$a_n = 2 \frac{2^n}{3}-\frac{1}{2}-\frac{(-1)^n}{6}=\lfloor \frac{2^{n+1}}{3} \rfloor.$$
It satisfies $a_{2n}=2a_{2n-1}$, $a_{2n+1}=2a_{2n}+1$. It counts the number of steps to change from a binary string of $n$ $0$'s to $n$ $1$'s using a Gray code (J. Stadler). This sequence has generating function
$$\frac{x}{(1+x)(1-x)(1-2x)}.$$ The ``second-occurrence'' sequence beginning $0,3,4,9,18,\ldots$ has generating function
$$\frac{x(3-2x-2x^2+2x^3)}{(1+x)(1-x)(1-2x)}.$$
The table above now suggests the following conjecture.
\begin{conjecture} The locations of the  occurrences of the digit $n$ in the sequence $s_n$ are given by those numbers whose base $2$ representation has exactly $n-1$ runs.
\end{conjecture}
We next examine a link between the locations of the $1$'s and the $-1$'s in the Jacobi sequence and the elements of the sequence $s_n$. It is clear that the elements appearing in $s_n$ are determined by sign changes in the Jacobi sequence. The locations of $1$'s appearing in the Jacobi sequence are indexed by \seqnum{A091072} which begins
$$0,1, 2, 4, 5, 8, 9, 10, 13, 16, 17, 18, 20, 21, 25, 26, 29, 32,\ldots.$$ This is the sequence $a_n$ such that the odd part of $a_n$ is of the form $4k+1$. The characteristic function of this sequence is given by the paper-folding sequence \seqnum{A014577}.
We then have the following.
\begin{conjecture} Let $a_n$ denote the sequence such that the odd part of $a_n$ is of the form $4k+1$. Then we have
$$a_n+s_{a_n}=2n+1.$$
\end{conjecture}
For instance, we have
$$\{0,1, 2, 4, 5, 8, 9, 10, 13\}+\{1, 2, 3, 3, 4, 3, 4, 5, 4\}=\{1,3,5,7,9,11,13,15,17\}.$$
The location of the $-1$'s in the Jacobi sequence is indexed by the sequence \seqnum{A091067} which begins
$$3, 6, 7, 11, 12, 14, 15, 19, 22, 23, 24, 27, 28,\ldots.$$ This is the sequence $b_n$ such that the odd part of $b_n$ is of the form $4k+3$. We then have the following.
\begin{conjecture} Let $b_n$ be the sequence such that the odd part of $b_n$ is of the form $4k+3$. Then
we have
$$b_n - s_{b_n}=2n+1.$$
\end{conjecture}
For instance, we have
$$\{3, 6, 7, 11, 12, 14, 15, 19, 22\}-\{2, 3, 2, 4, 3, 3, 2, 4, 5\}=\{1,3,5,7,9,11,15,17\}.$$

We note that since this paper was written, the three conjectures above have now been proven by Allouche and Shallit \cite{Conjectures}.
\section{The Rueppel sequence - introduction}
The Rueppel sequence $r_n$ is the sequence \seqnum{A036987} that begins
$$1, 1, 0, 1, 0, 0, 0, 1, 0, 0, 0,\ldots,$$ with generating function given by
$$r(x)=\sum_{n \ge 0} x^{2^{n-1}}=1+x+x^3+x^7+x^{15}+\cdots.$$
The Catalan numbers \seqnum{A000108} are defined by
$$C_n = \frac{1}{n+1}\binom{2n}{n}.$$
These two sequences are related by
$$r_n = C_n \bmod 2,$$ and we would thus expect some of the properties of the sequence $r_n$ to be ``inherited'' from the Catalan number sequence. An example is given by expressing their generating functions as Jacobi continued fractions. It is well known that the generating function
$$c(x)=\frac{1-\sqrt{1-4x}}{2x}$$ of the Catalan numbers can be expressed as
$$c(x)=\cfrac{1}{1-x-
\cfrac{x^2}{1-2x-
\cfrac{x^2}{1-2x-
\cfrac{x^2}{1-2x-\cdots}}}}.$$
Thus the Catalan numbers are determined by the two sequences
$$1,1,1,\ldots$$
and
$$1,2,2,2,2,\ldots.$$
These sequences are called the Jacobi (or Markov) parameters of the sequence.

The Jacobi parameters for the Rueppel sequence are given by the sequences \cite{Bacher, Cigler}
$$-1,-1,-1,-1,\ldots$$ and
$$1, -2, 0, 0, 2, 0, -2, 0, 2, -2, 0,\ldots.$$ This latter sequence has generating function
$$\frac{1-2x-x^2}{1+x^2}+\frac{2}{x}\sum_{k \ge 0} \frac{x^{3\cdot 2^{k-1}}}{\prod_{i=0}^k 1+x^{2^i}}.$$
It given by $-$\seqnum{A110036}$(n+1)$.

The generating function $r(x)$ of $r_n$ then has a continued fraction expression that begins
$$\cfrac{1}{1-x+
\cfrac{x^2}{1+2x+
\cfrac{x^2}{1+
\cfrac{x^2}{1+
\cfrac{x^2}{1-2x+
\cfrac{x^2}{1-\cdots}}}}}}.$$
The sequence $-1,-1,-1,-1,\ldots$ is the negative of the same sequence for the Catalan numbers.

For a Jacobi continued fraction, we shall use the term \emph{$\alpha$-sequence} for the sequence of coefficients of $x$, and the term \emph{$\beta$-sequence} for the coefficients of the $x^2$ term.
Then \cite{Bacher, Cigler}
the $\alpha$-sequence of the Rueppel numbers is given by $-\seqnum{A110036}(n+1)$.

According to $\seqnum{A110036}$, the sequence
$a_n=\seqnum{A110036}$ gives the constant terms of the partial quotients of the continued fraction expansion of $$1+\sum_{n \ge 0} \frac{1}{x^{2^n}},$$
where each partial quotient has the form $\{x + a_n\}$ after the initial constant term of $1$. (Paul D. Hanna, Ralph Stephan).

We note that we can also express the Rueppel sequence $r_n$ as
$$r_n = \left(\sum_{k=0}^n \binom{2n-k}{n}\right) \bmod 2.$$
We also have
$$r_n^{(+)} = \left(\sum_{k=0}^n \binom{2n}{n+k}\right) \bmod 2,$$ where $r_n^{(+)}$ is the augmented Rueppel sequence, that is, the Rueppel sequence with a pre-pended $1$.

For a number triangle $T_{n,k}$, the triangles $T_{2n-k,n}$ and $T_{2n,,n+k}$ are called respectively the vertical half and the horizontal half \cite{Halves}.

Of interest as well is the sequence $r_{n+1}$. This sequence has the interesting property that it gives the aeration of the Rueppel sequence $r_n$. In terms of generating functions, we have
$$ r(x^2)=\frac{r(x)-1}{x}\quad \text{or}\quad r(x)-1=xr(x^2).$$

We have the following result \cite{Bacher, Cigler} for which we give an independent proof.
\begin{proposition} The Hankel transform $h_n^*$ of the shifted Rueppel sequence $r_{n+1}$ is given by
$$  h_0^*=h_0, h_{2n}^*=h_{n-1}^* h_n, h_{2n+1}^*=h_n^*h_n,$$ where
$h_n=(-1)^{\binom{n+1}{2}}$ gives the Hankel transform of $r_n$.
\end{proposition}
\begin{proof}
The Hankel transform of an aerated sequence involves the terms of the Hankel transform of the original sequence $h_n$ and those of Hankel transform of the aerated sequence, $h_n^*$. Specifically, the Hankel transform of an aerated sequence begins \cite{Wall}
$$ h_0, h_0^* h_0, h_0^* h_1, h_1^* h_1, h_1^* h_2, h_2^* h_2, h_2^*h_3,\ldots.$$ In this case, we can equate these terms to
$$h_0^*, h_1^*, h_2^*, h_3^*, h_4^*, h_5^*, h_6^*,\ldots.$$
The result now follows.
\end{proof}
We remark that the sequences $r_n$ and $r_{n+1}$ are among a small number of basic sequences whose Hankel transforms can be fully described.

It is known that the Hankel transform $h_n$ of the paper-folding sequence is (modulo $2$) periodic of period $10$ \cite{Fu}. We have
$$h_n \bmod 2= \overline{1, 1, 1, 0, 0, 1, 0, 0, 1, 1}.$$

Although we shall not use it in this note, we give the example of the period doubling sequence $\cite{PD}$ for which there is a description of its Hankel transform.
\begin{example} The period doubling sequence is the sequence whose generating function can be given by
$$\frac{1}{x}\sum_{k=0}^{\infty} \frac{(-1)^k x^{2^k}}{1-x^{2^k}}.$$
This is \seqnum{A035263}, which begins
$$1, 0, 1, 1, 1, 0, 1, 0, 1, 0, 1, 1, 1, 0, 1, 1,\ldots.$$
Its Hankel transform $h_n$ is \seqnum{A265025}, which begins
$$1, 1, -1, -3, 1, 1, -1, -15, 1, 1, -1, -3, 1, 1, -9, -495, 9, 1,\ldots.$$
We then have \cite{PD}
$$h_{2^n}=-J_{n+1}\prod_{3 \le i \le n}J_i^{2^{n-i}},$$ where
$$J_n = \frac{2^n}{3}-\frac{(-1)^n}{3}$$ is the $n$-th Jacobsthal number. The sequence $J_n$ \seqnum{A001045} has generating function
$$\frac{x}{1-x-2x^2}.$$
\end{example}

\section{Modifying the Rueppel sequence and Hankel transforms}
The Catalan numbers $C_n$ begin
$$1,1,2,5,14,42,\ldots.$$
The Hankel transform of this sequence, that is, the sequence of determinants $|C_{i+j}|_{\{0 \le i,j, \le n\}}$ is given by the sequence
$$1,1,1,1,\ldots,$$ determined by the $\beta$-sequence of the Catalan numbers. The Hankel transform of $r_n$ is given by the sequence
$$1, -1, -1, 1, 1, -1, -1, 1, 1, -1, -1,\ldots,$$ with general term $(-1)^{\binom{n+1}{2}}$. This again follows from the form of the $\beta$-sequence of this sequence.

For the two variants of the Catalan numbers given below,
$$1,1,1,2,5,14,42,\ldots,$$ and
$$1,-1,-1,-2,-5,-14,-42,\ldots,$$ we find that both their Hankel transforms are given by
$$1, 0, -1, -2, -3, -4, -5, -6, -7, -8, \ldots.$$ Taken modulo $2$, we get the sequence
$$1,0,1,0,1,0,1,0,1,\ldots.$$
We are thus prompted to look at the analagous variants of the Rueppel sequence.

We consider first the sequence $r_n^{(-)}$ with generating function $1-x r(x)$ that begins
$$1,-1, -1, 0, -1, 0, 0, 0, -1, 0, 0, 0,\ldots.$$
Note that we have
$$r_{2n}^{(-)}=r_n^{(-)}.$$
We have the following conjectures about the Hankel transform of this sequence.
\begin{conjecture}
The Hankel transform $h_n$ of the sequence $r_n^{(-)}$ begins
$$1, -2, 3, 2, -3, 4, 3, 2, -3, 4, -5,\ldots.$$
Then $|h_n|=s_n$ is the sequence $\seqnum{A088748}$, that is,
$$|h_n|=1+\sum_{k=0}^{n} j_k,$$ where
$$j_n=\left(\frac{-1}{n}\right)$$ is the Jacobi symbol.
\end{conjecture}
Thus we conjecture that $|h_n|$ coincides with \seqnum{A088748}. We shall use the notation $\tilde{s}_n$ to denote the Hankel transform of $r_n^{(-1)}$.
\begin{conjecture} Let $\tilde{s}_n$ be the Hankel transform of $r_n^{(-1)}$. Then $$|\tilde{s}_n|=s_n.$$
\end{conjecture}

The sequence $\sigma_n$ of signs given by
$$\sigma_n=\frac{h_n}{|h_n|}$$ begins
$$1, -1, 1, 1, -1, 1, 1, 1, -1, 1, -1,\ldots.$$ This sequence is of independent interest.
We have the following conjecture, where we define the paper-folding sequence $p_n$ \seqnum{A014577} by
$$p_n=\frac{1+j_{n+1}}{2}.$$
\begin{conjecture}
We have $$\frac{|\sigma_{n+1}-\sigma_n|}{2}=p_n.$$
\end{conjecture}
In fact, we can say more than this.
\begin{conjecture}
We have $$\frac{|\sigma_n+\sigma_{n+1}|}{2}=\bar{p}_n.$$
\end{conjecture}
The sequence $\frac{\sigma_n+\sigma_{n+1}}{2}$ is a signed version of \seqnum{A014707}. It begins
$$0, 0, 1, 0, 0, 1, 1, 0, 0, 0, -1, -1, 0, 1, 1, 0, 0, 0, -1, 0, 0, -1,
-1, -1, 0, 0, -1,\ldots.$$

The positions $n$ where $\sigma_n \sigma_{n+1}=-1$ are indexed by the sequence that begins
$$0, 1, 3, 4, 7, 8, 9, 12, 15, 16, 17, 19, 20, 24, 25, 28, 31, 32, \ldots.$$
Incrementing by $1$, we get the sequence \seqnum{A091072}
$$	1, 2, 4, 5, 8, 9, 10, 13, 16, 17, 18, 20, 21, \ldots$$ whose characteristic function as we have seen is the paper-folding sequence $p_n$.

The sequence $\tilde{s}_n-\sigma_n$ begins
$$0, -1, 2, 1, -2, 3, 2, 1, -2, 3, -4, -3, -2, 3, 2, 1,\ldots.$$ Our conjecture concerning $\tilde{s}_n$ implies that $\sigma_n \tilde{s}_n=s_n$, or equivalently, $\tilde{s}_n=\sigma_n s_n$, and hence we would have
$$\tilde{s}_n - \sigma_n= \sigma_n s_n-\sigma_n=\sigma_n (s_n-1)=\sigma_n g_n.$$ We state this as a conjecture.
\begin{conjecture}
We have $$\tilde{s}_n - \sigma_n=\sigma_n g_n.$$
\end{conjecture}
Recalling that $g_n$ is the sequence of partial sums of the Jacobi sequence, it is interesting to find the first differences of $\tilde{s}_n-\sigma_n$, which we call $\tilde{j}_n$. Thus
$$\tilde{j}_n=\tilde{s}_n-\sigma_n.$$
The sequence $\tilde{j}_n$  begins
$$-1, 3, -1, -3, 5, -1, -1, -3, 5, -7, 1, 1, 5, -1, -1, -3, 5, -7, 1,
7, -9, 1, 1, 1, 5,\ldots.$$
We find (empirically) that the positions of $\pm 1$ in this sequence are indexed by
$$0, 2, 5, 6, 10, 11, 13, 14, 18, 21, 22, 23, 26, 27, 29, 30, 34, 37,\ldots,$$
or \seqnum{A255068}=\seqnum{A091067}$-1$. Thus annulling the elements of $\tilde{j}_n$ not equal to $\pm1 $ and taking absolute values gives the sequence that begins
$$1, 0, 1, 0, 0, 1, 1, 0, 0, 0, 1, 1, 0, 1, 1, 0, 0, 0, 1, 0, 0, 1, 1,
1, 0, 0, 1, 1, 0, 1, 1, 0, 0, 0, 1,$$ which is essentially \seqnum{A014707}. The complement of this sequence
begins
$$0, 1, 0, 1, 1, 0, 0, 1, 1, 1, 0, 0, 1, 0, 0, 1, 1, 1, 0, 1, 1, 0, 0,
0, 1,\ldots,$$ which coincides with the paper-folding variant \seqnum{A082410}.

The characteristic sequence of the terms of $\tilde{j}_n$ which are not equal to $\pm 1$ is the paper-folding sequence $p_{n+1}$.

We recall that the positions of $+1$ in the Jacobi sequence are given by the sequence $a_n$. We find that the sequence $\tilde{s}_{a_n}$ begins
$$1, -2, 3, -3, 4, -3, 4, -5, 4, -3, 4, -5, 5, -6, 4, -5, 4, -3, 4, -5,
5, -6, 5, -6, 7,\ldots,$$ with alternating signs. Thus we can conjecture that
$$\sigma_{a_n}=(-1)^n.$$
We do not find such regularity with respect to locations based on the $-1$ terms of the Jacobi sequence, defined by the sequence $b_n$. Thus we have that $\tilde{s}_{b_n}$ begins
$$2, 3, 2, -4, -3, 3, 2, -4, -5, -4, -3,\ldots.$$ The corresponding sequence of signs $\sigma_{b_n}$ begins
$$1, 1, 1, -1, -1, 1, 1, -1, -1, -1, -1, -1, -1, 1, 1, -1,\ldots.$$
We are unable to say anything about this sequence.

We now move to the Hankel transform of the sequence $r_n^{(+)}$ with generating function $1+xr(x)$ obtained by pre-pending a $1$ to the Rueppel sequence. Thus this sequence begins
$$1,1, 1, 0, 1, 0, 0, 0, 1, 0, 0, 0,\ldots,$$ with generating function
$$1+xr(x)=1+x+x^2+x^4+x^8+x^{16}+\ldots.$$
We then have the following conjecture. For this, we recall the notation
$$g_n=\sum_{k=0}^{n} j_k$$ to denote the sequence \seqnum{A005811} of the number of runs in the binary expansion of $n$ ($n \ge 0$) which begins
$$0, 1, 2, 1, 2, 3, 2, 1, 2, 3, 4, 3, 2, 3, 2,\ldots.$$
Then the sequence $g_n-1$ begins
$$-1, 0, 1, 0, 1, 2, 1, 0, 1, 2, 3, 2, 1, 2, 1, 0, 1,\ldots.$$
\begin{conjecture}
The Hankel transform $h_n$ of the sequence $r_n^{(+)}$ begins
$$1, 0, -1, 0, 1, 2, -1, 0, 1, 2, 3, -2, 1, 2, -1, 0, 1, \ldots.$$
For $n>0$, we have
$$|h_n|=g_n -1.$$
\end{conjecture}

The Catalan matrix $(c(x), xc(x)^2)=(c(x), c(x)-1)$ is of combinatorial importance. Thus it behoves us to look at the arrays $(r(x), r(x)-1)$ and $(r(x), xr(x)^2)$. The expansion of the difference $r(x)-1-xr(x)^2$ begins
$$0, 0, -2, 0, -2, -2, 0, 0, -2, -2, 0, -2, 0, 0, 0, 0, -2, -2, 0, -2, 0,\dots,$$
which modulo two gives $0,0,0,\ldots$, but the arrays $(r(x), r(x)-1)$ and $(r(x), xr(x)^2)$ are distinct.

The array $(r(x), r(x)-1)$ begins
$$\left(
\begin{array}{cccccccc}
 1 & 0 & 0 & 0 & 0 & 0 & 0 & 0 \\
 1 & 1 & 0 & 0 & 0 & 0 & 0 & 0 \\
 0 & 1 & 1 & 0 & 0 & 0 & 0 & 0 \\
 1 & 1 & 1 & 1 & 0 & 0 & 0 & 0 \\
 0 & 2 & 2 & 1 & 1 & 0 & 0 & 0 \\
 0 & 0 & 3 & 3 & 1 & 1 & 0 & 0 \\
 0 & 1 & 1 & 4 & 4 & 1 & 1 & 0 \\
 1 & 1 & 3 & 3 & 5 & 5 & 1 & 1 \\
\end{array}
\right).$$ Its row sums begin
$$1, 2, 2, 4, 6, 8, 12, 20, 30, 44, 68, 104, 156, 236, 360, 548, 830,\ldots,$$
with Hankel transform equal to $2^n (-1)^{binom{n+1}{2}}$. The Jacobi parameters begin
$$[2, -2, 0, 0, 2, 0, -2, 0, 2, -2,\ldots$$ and
$$-2,-1,-1,-1,-1,\ldots$$ respectively for the $\alpha$ and $\beta$ sequences. When $(r(x), r(x)-1)$ is taken modulo $2$, the row sums give Gould's sequence. The inverse array, when taken modulo $2$, has row sums equal to the Farey denominators \seqnum{A007306}$(n+1)$. We note that the halves of the array, taken modulo $2$, have row sums that begin
$$1, 2, 2, 2, 2, 4, 2, 4, 2, 4, 4, 6, 2, 4, 4, 8, 2,\ldots$$ and
$$1, 2, 2, 3, 3, 4, 5, 4, 4, 5, 5, 6, 7, 8, 7, 5, 5, 7,\ldots$$ for the horizontal and vertical halves, respectively.
The horizontal half sequence $a_n$ beginning $1, 2, 2, 2, 2, 4, 2, 4, 2, 4, 4,\ldots$ appears amenable to the following analysis.
We have $a_0=1$, $a_1=2$, and for $n>1$, if $n$ is even, $a_n=a_{\frac{n}{2}}$, while if $n$ is odd, we have
$$a_n =2 a_{\frac{n-1}{2}}-2 If[\log_4\left(3 \frac{n-1}{2}+1\right) \,\text{is an integer}, 1,0].$$

\section{The sequence $\tilde{s}_n$}
In this section, we look at the distribution of the negative and positive digits in the sequence $\tilde{s}_n$. We recall that this sequence begins
$$1, -2, 3, 2, -3, 4, 3, 2, -3, 4, -5, -4, -3,\ldots.$$
The negative elements in the sequence $\tilde{s}_n$ occur at the following positions
$$1, 4, 8, 10, 11, 12, 16, 18, 19, 21, 22, 23, 24, 26, 27, 28, 32, 34,\ldots,$$ while the positive digits occur at the positions given by the complementary sequence
$$0, 2, 3, 5, 6, 7, 9, 13, 14, 15, 17, 20, 25, 29, 30, 31, 33, 36, 40, 42, \ldots.$$
We first look at the occurrences of the positive digits in $\tilde{s}_n$, which we summarize in the table below.
\begin{center}
\begin{tabular}{|c||c|c|c|c|c|c|}
\hline $n$ & $1^{\text{st}}$  & $2^{\text{nd}}$ & $3^{\text{rd}}$ & $4^{\text{th}}$ & Axxxxxx & Comment\\
\hline\hline $2$ & $3$ & $7$ & $15$ & $31$ & $2*$\seqnum{A000225}$+1$ & $2(1$ run)$+1$\\
\hline $3$ & $2$ & $6$ & $14$ & $30$ & $2*$\seqnum{A000225} & $2(1$ run$)$ \\
\hline $4$ & $5$ & $9$ & $13$ & $17$ & $2*$\seqnum{A043569}$+1$ & $2(2$ runs$)+1$ \\
\hline $5$ & $20$ & $36$ & $40$ & $44$ & $2*$\seqnum{A043571} & $2(4$ runs$)$ \\
\hline $6$ & $43$ & $75$ & $83$ & $87$ & $2*$\seqnum{A043572}$+1$ & $2(5$ runs$)+1$\\
\hline $7$ & $42$ & $74$ & $82$ & $86$ & $2*$\seqnum{A043572} & $2(5$ runs) \\
\hline $8$ & $85$ & $149$ & $165$ & $169$ & $2*$\seqnum{A043573}$+1$ & $2(6$ runs$)+1$\\
\hline $9$ & $340$ & $596$ & $$ & $$ & $2*$\seqnum{A043575} & $2(8$ runs$)$ \\
\hline $10$ & $$ & & & & \\
\hline
\end{tabular}
\end{center}
We can analyze the ``first occurence'' sequence $3,2,5,20,\ldots$ as follows. We consider the sequence \seqnum{A005187}, which can be defined as the sequence $t_n$ with
$$t_n=2n-\seqnum{A000120}(n),$$ where
\seqnum{A000120} counts the number of $1$'s in the binary expansion of $n$. This sequence begins
$$0, 1, 3, 4, 7, 8, 10, 11, 15, 16, 18, 19,\ldots.$$ We let $b_n$ be the sequence \seqnum{A042964}, of numbers that are congruent to $2$ or $3$ modulo $4$.
Then the sequence $3,2,5,20,\ldots$ is given by
$$J_{a_{\frac{n+4}{2}}}\frac{1-(-1)^n}{2}+\left(J_{b_{\frac{n+1}{2}}}-1\right)\frac{1+(-1)^n}{2},$$ where
$J_n$ is the $n$-th Jacobsthal number.

Turning now to the negative entries, we have the following table.
\begin{center}
\begin{tabular}{|c||c|c|c|c|c|c|c|}
\hline $n$ & $1^{\text{st}}$  & $2^{\text{nd}}$ & $3^{\text{rd}}$ & $4^{\text{th}}$ & Axxxxxx & Comment \\
\hline\hline $-2$ & $1$ & $$ & $$ & $$ &  & \\
\hline $-3$ & $4$ & $8$ & $12$ & $16$ & $2*$\seqnum{A043569} & $2(2$ runs$)$ \\
\hline $-4$ & $11$ & $19$ & $23$ & $27$ & $2*$\seqnum{A043570}$+1$ & $2(3$ runs$)+1$\\
\hline $-5$ & $10$ & $18$ & $22$ & $26$ & $2*$\seqnum{A043570} & $2(3$ runs$)$\\
\hline $-6$ & $21$ & $37$ & $41$ & $45$ & $2*$\seqnum{A043571}$+1$ & $2(4$ runs$)+1$ \\
\hline $-7$ & $84$ & $148$ & $164$ & $168$ & $2*$\seqnum{A043573} & $2(6$ runs$)$\\
\hline $-8$ & $171$ & $299$ & $331$ & $339$ & $2*$\seqnum{A043574}$+1$ & $2(7$ runs$)+1$\\
\hline $-9$ & $170$ & $298$ & $330$ & $338$ & $2*$\seqnum{A043574} & $2(7$ runs$)$ \\
\hline $-10$ & $341$ & $597$ & $$ & $$ & $2*$\seqnum{A043575}$+1$ & $2(8$ runs$)+1$\\
\hline
\end{tabular}
\end{center}
We note that the first occurrence sequence of the negative entries which begins $$0,1,4,11,10,21,84,170,341,\ldots$$ is determined by the Jacobsthal numbers. We can conjecture the following expression for the $n$-th term of this sequence.
$$\frac{1}{2}\left(J_{n+2}\left(1+(-1)^{\frac{n+1}{2}}\right)+J_{n+1}\left(1-(-1)^{\frac{n+1}{2}}\right)\right)(1-(-1)^n)+$$
$$\quad\frac{1}{2}\left(J_{n+2}\left(1-(-1)^{\frac{n}{2}}\right)+J_{n+1}\left(1+(-1)^{\frac{n}{2}}\right)\right)(1+(-1)^n)-\frac{1+(-1)^n}{2}.$$
The corresponding generating function is then given by
$$\frac{x(1+2x+7x^2-4x^3+12x^4-8x^5)}{1-2x+4x^2-8x^3-x^4+2x^5-4x^6+8x^7}.$$ The denominator factorizes as
$$(1-x^2)(1+x^2)(1-2x)(1+4x^2).$$
An alternative approach to the sequence $1,4,11,10,\ldots$ is to  consider the sequence
$$a_n=2n+\frac{1+(-1)^n}{2},$$ or \seqnum{A042963} of numbers that are congruent to $1$ or $2$ modulo $4$, and the sequence
$$b_n=2n-\frac{1-(-1)^n}{2},$$ or \seqnum{A042948} of numbers that are congruent to $0$ or $1$ modulo $4$. Then the sequence $1,4,11,\ldots$ is given by
$$J_{a_{\frac{n+2}{2}}}\frac{1+(-1)^n}{2}+J_{b_{\frac{n+3}{2}}}\frac{1-(-1)^n}{2}-\frac{1-(-1)^n}{2}.$$
We next look at the locations of the negative numbers $-n, n >0$ relative to their position in the sequences that give the locations of $n$ in $s_n$. For instance, the locations of the entry $3$ in $s_n$ are given by the sequence
$$2, 4, 6, 8, 12, 14, 16, 24, 28, 30, 32, 48, 56, 60, 62, 64, \ldots.$$
The locations of $-3$ in $\tilde{s}_n$ are given by
$$4, 8, 12, 16, 24, 28, 32, 48, 56, 60, 64, 96, 112, 120, 124, 128,
192, 224,\ldots,$$ which is a subsequence of the first sequence.

Relative to the $3$ in $s_n$ sequence, this latter sequence is indexed by
$$1, 3, 4, 6, 7, 8, 10, 11, 12, 13, 15, 16, 17, 18, 19, 21, 22, 23,\ldots.$$
This is \seqnum{A007401}. This sequence $t_n$ has many interesting properties. For instance, the sequence $t_n-n$ begins
$$1, 2, 2, 3, 3, 3, 4, 4, 4, 4, 5, 5, 5, 5, 5, 6, 6, 6, 6, 6, 6,\ldots,$$
which is \seqnum{A002024} (``$n$ appears $n$ times''). Thus
$$t_n-n=\lfloor \frac{1+\sqrt{1+8n}}{2} \rfloor, \quad \text{or} \quad t_n=n+\lfloor \frac{1+\sqrt{1+8n}}{2} \rfloor.$$ The generating function of $t_n$ is given by
$$t(x)=\frac{x}{(1-x)^2}+\frac{1}{1-x} \prod_{k=1}^{\infty} \frac{1-x^{2k}}{1-x^{2k-1}}.$$
Alternatively, we can express this as
$$t(x)=\frac{x}{(1-x)^2}+\frac{1}{1-x} \theta_2(0,\sqrt{x})/(2x^{\frac{1}{8}}).$$
Here, $\theta_2(z,x)$ is the Jacobi elliptic function $\theta_2$ \cite{WW}.
Similarly, the sequence $t_n-n-1$ begins
$$0, 1, 1, 2, 2, 2, 3, 3, 3, 3, 4, 4, 4, 4, 4, 5, 5, 5, 5, 5, 5,\ldots.$$ This is \seqnum{A003056} (``$n$ appears $n+1$ times''). We then have
$$t_n-n-1=\lfloor \frac{\sqrt{8n+1}-1}{2} \rfloor,\quad \text{or}\quad t_n=1+n+\lfloor \frac{\sqrt{8n+1}-1}{2} \rfloor.$$
We can express the sequence $t_n$ in triangle form $T$ as the triangle that begins
$$\left(
\begin{array}{ccccccc}
 1 & 0 & 0 & 0 & 0 & 0 & 0 \\
 3 & 4 & 0 & 0 & 0 & 0 & 0 \\
 6 & 7 & 8 & 0 & 0 & 0 & 0 \\
 10 & 11 & 12 & 13 & 0 & 0 & 0 \\
 15 & 16 & 17 & 18 & 19 & 0 & 0 \\
 21 & 22 & 23 & 24 & 25 & 26 & 0 \\
 28 & 29 & 30 & 31 & 32 & 33 & 34 \\
\end{array}
\right).$$
The general $(n,k)$ term of this matrix is then given by
$$t_{n,k}=\frac{(n+1)(n+2)+2k}{2}=\frac{n(n+3)}{2}+k+1.$$
We can factorize this matrix as the following product.
$$\left(
\begin{array}{ccccccc}
 1 & 0 & 0 & 0 & 0 & 0 & 0 \\
 1 & 1 & 0 & 0 & 0 & 0 & 0 \\
 1 & 1 & 1 & 0 & 0 & 0 & 0 \\
 1 & 1 & 1 & 1 & 0 & 0 & 0 \\
 1 & 1 & 1 & 1 & 1 & 0 & 0 \\
 1 & 1 & 1 & 1 & 1 & 1 & 0 \\
 1 & 1 & 1 & 1 & 1 & 1 & 1 \\
\end{array}
\right)\cdot \left(
\begin{array}{ccccccc}
 1 & 0 & 0 & 0 & 0 & 0 & 0 \\
 2 & 4 & 0 & 0 & 0 & 0 & 0 \\
 3 & 3 & 8 & 0 & 0 & 0 & 0 \\
 4 & 4 & 4 & 13 & 0 & 0 & 0 \\
 5 & 5 & 5 & 5 & 19 & 0 & 0 \\
 6 & 6 & 6 & 6 & 6 & 26 & 0 \\
 7 & 7 & 7 & 7 & 7 & 7 & 34 \\
\end{array}
\right).$$
We can represent the triangle $T$ as the element-wise sum matrix of the following triangle
$$\left(
\begin{array}{cccccc}
 1 & 0 & 0 & 0 & 0 & 0 \\
 2 & 1 & 0 & 0 & 0 & 0 \\
 2 & 1 & 1 & 0 & 0 & 0 \\
 2 & 1 & 1 & 1 & 0 & 0 \\
 2 & 1 & 1 & 1 & 1 & 0 \\
 2 & 1 & 1 & 1 & 1 & 1 \\
\end{array}
\right),$$
which is the triangle of the sequence that begins
$$1, 2, 1, 2, 1, 1, 2, 1, 1, 1, 2, 1, 1, 1, 1, 2, 1, 1, 1, 1, 1, 2, 1, 1, 1, \ldots.$$
This sequence then has generating function
$$\frac{x}{1-x}+ \theta_2(0,\sqrt{x})/(2x^{\frac{1}{8}}).$$
The $t_n$ triangle $T$ has the following factorization.
$$\left(
\begin{array}{cccccc}
 1 & 0 & 0 & 0 & 0 & 0 \\
 2 & 1 & 0 & 0 & 0 & 0 \\
 2 & 1 & 1 & 0 & 0 & 0 \\
 2 & 1 & 1 & 1 & 0 & 0 \\
 2 & 1 & 1 & 1 & 1 & 0 \\
 2 & 1 & 1 & 1 & 1 & 1 \\
\end{array}
\right)\cdot \left(
\begin{array}{cccccc}
 1 & 0 & 0 & 0 & 0 & 0 \\
 1 & 4 & 0 & 0 & 0 & 0 \\
 3 & 3 & 8 & 0 & 0 & 0 \\
 4 & 4 & 4 & 13 & 0 & 0 \\
 5 & 5 & 5 & 5 & 19 & 0 \\
 6 & 6 & 6 & 6 & 6 & 26 \\
\end{array}
\right).$$
We can also represent the sequence
$$1, 2, 1, 2, 1, 1, 2, 1, 1, 1, 2, 1, 1, 1, 1, 2, 1, 1, 1, 1, 1, 2, 1, 1, 1, \ldots$$ as a square array read by anti-diagonals. This takes the following form.
$$\left(
\begin{array}{cccccc}
 1 & 1 & 1 & 1 & 1 & 1 \\
 2 & 1 & 1 & 1 & 1 & 1 \\
 2 & 1 & 1 & 1 & 1 & 1 \\
 2 & 1 & 1 & 1 & 1 & 1 \\
 2 & 1 & 1 & 1 & 1 & 1 \\
 2 & 1 & 1 & 1 & 1 & 1 \\
\end{array}
\right).$$
Multiplying this by the inverse binomial matrix $B^{-1}$ on the left,  we obtain the simple array
$$\left(
\begin{array}{cccccc}
 1 & 1 & 1 & 1 & 1 & 1 \\
 1 & 0 & 0 & 0 & 0 & 0 \\
 -1 & 0 & 0 & 0 & 0 & 0 \\
 1 & 0 & 0 & 0 & 0 & 0 \\
 -1 & 0 & 0 & 0 & 0 & 0 \\
 1 & 0 & 0 & 0 & 0 & 0 \\
\end{array}
\right).$$
This array has generating function
$$f(x,y)=\frac{1}{1-y}+\frac{x}{1+x}.$$ Thus the previous array will have generating function
$$\frac{1}{1-x} f\left(\frac{x}{1-x},y\right)=\frac{1-x(y-1)}{(1-x)(1-y)}.$$
We deduce that the triangle that begins
$$\left(
\begin{array}{cccccc}
 1 & 0 & 0 & 0 & 0 & 0 \\
 2 & 1 & 0 & 0 & 0 & 0 \\
 2 & 1 & 1 & 0 & 0 & 0 \\
 2 & 1 & 1 & 1 & 0 & 0 \\
 2 & 1 & 1 & 1 & 1 & 0 \\
 2 & 1 & 1 & 1 & 1 & 1 \\
\end{array}
\right)$$ has generating function given by
$$\frac{1-x(xy-1)}{(1-x)(1-xy)}.$$
In a similar manner we can represent the sequence \seqnum{A007401} by the square array that begins
$$\left(
\begin{array}{cccccc}
 1 & 4 & 8 & 13 & 19 & 26 \\
 3 & 7 & 12 & 18 & 25 & 33 \\
 6 & 11 & 17 & 24 & 32 & 41 \\
 10 & 16 & 23 & 31 & 40 & 50 \\
 15 & 22 & 30 & 39 & 49 & 60 \\
 21 & 29 & 38 & 48 & 59 & 71 \\
\end{array}
\right).$$
Multiplying again by the inverse binomial matrix, we can encode this as the array that begins
$$\left(
\begin{array}{cccccc}
 1 & 4 & 8 & 13 & 19 & 26 \\
 2 & 3 & 4 & 5 & 6 & 7 \\
 1 & 1 & 1 & 1 & 1 & 1 \\
 0 & 0 & 0 & 0 & 0 & 0 \\
 0 & 0 & 0 & 0 & 0 & 0 \\
 0 & 0 & 0 & 0 & 0 & 0 \\
\end{array}
\right).$$
Thus, for example, the second column $4,7,11,16,\ldots$ has generating function given by the binomial transform of the polynomial $4+3x+x^2$. This is given by $\frac{2x^2-5x+4}{(1-x)^3}$. In general, the $n$-th column will have generating function equal to
$$\frac{n(n+3)}{2}x^2-((n+2)^2-4)x+\frac{(n+1)(n+4)}{2}-1.$$
We may also consider the matrix
$$B^{-1}\left(
\begin{array}{cccccc}
 1 & 4 & 8 & 13 & 19 & 26 \\
 3 & 7 & 12 & 18 & 25 & 33 \\
 6 & 11 & 17 & 24 & 32 & 41 \\
 10 & 16 & 23 & 31 & 40 & 50 \\
 15 & 22 & 30 & 39 & 49 & 60 \\
 21 & 29 & 38 & 48 & 59 & 71 \\
\end{array}
\right)(B^{-1})^T.$$ This yields the matrix
$$\left(
\begin{array}{cccccc}
 1 & 3 & 1 & 0 & 0 & 0 \\
 2 & 1 & 0 & 0 & 0 & 0 \\
 1 & 0 & 0 & 0 & 0 & 0 \\
 0 & 0 & 0 & 0 & 0 & 0 \\
 0 & 0 & 0 & 0 & 0 & 0 \\
 0 & 0 & 0 & 0 & 0 & 0 \\
\end{array}
\right)$$ with generating function
$$f(x,y)=(1+x)^2+y(3x+1)+y^2.$$
Thus the square matrix representing the sequence \seqnum{A007401} will have its generating function given by
$$\frac{1}{(1-x)(1-y)}f\left(\frac{x}{1-x}, \frac{y}{1-y}\right),$$ which gives
$$\frac{1+y-y^2-(5-3y)xy-(y-2)x^2y}{(1-x)^3(1-y)^3}.$$ As a triangle read by rows, \seqnum{A007401} will then have generating function
$$\frac{1+xy-x^2y^2-(5-3xy)x^2y-(xy-2)x^3y}{(1-x)^3(1-xy)^3}.$$

The relative locations of  $+3$ in the sequence $\tilde{s}_n$ are given by the sequence which is complementary to \seqnum{A007401}. This sequence begins
$$0, 2, 5, 9, 14, 20, 27, 35, 44, 54, 65, 77,\ldots.$$
This is \seqnum{A000096}, with general term $$\frac{n(n+3)}{2}=n+\binom{n+1}{2}=\sum_{k=0}^{n-1} (k+2).$$ Its generating function is given by
$$ \frac{x(2-x)}{(1-x)^3}.$$
Thus this sequence is the partial sum sequence of
$$0,2,3,4,5,6,7,8,\ldots,$$ which is itself the partial sum sequence of
$$0,2,1,1,1,1,1,1,\ldots.$$
We now look at the case of $\pm 4$. The locations of $\pm 4$ in the sequence $\tilde{s}_n$ are given by the locations of $4$ in $s_n$. These are specified by the sequence that begins
$$5, 9, 11, 13, 17, 19, 23, 25, 27, 29, 33, 35, 39, 47,\ldots.$$
The locations of $(-4)$ are given by the subsequence that begins
$$11, 19, 23, 27, 35, 39, 47, 51, 55, 59, 67, 71, 79, 95,\ldots.$$
Relative to the $\pm 4$ sequence, this sequence is specified by the index sequence
$$2,5,6,8,11,12,13,15,16,18,22,23,24,26,\ldots.$$ The first differences of this sequence are given by \seqnum{A073645}, which begins
$$2, 3, 1, 2, 3, 1, 1, 2, 1, 2, 3, 1, 1, 1, 2, 1, 1, 2, 1, 2, 3, 1, 1,\ldots.$$
We can arrange this sequence in the irregular triangle that begins
$$
\begin{array}{cccccccccc}
 2 & 3 &  &  &  &  &  & &  &  \\
 1 & 2 & 3 &  &  &  &  &  &  &  \\
 1 & 1 & 2 & 1 & 2 & 3 &  &  &  &  \\
 1 & 1 & 1 & 2 & 1 & 1 & 2 & 1 & 2 & 3 \\
\end{array}
$$
Summing, we get the array
$$
\begin{array}{cccccccccc}
 2 & 5 &  &  &  &  &  &  &  &  \\
 6 & 8 & 11 &  &  &  &  &  &  &  \\
 12 & 13 & 15 & 16 & 18 & 21 &  &  &  &  \\
 22 & 23 & 24 & 26 & 27 & 28 & 30 & 31 & 33 & 36 \\
\end{array}
$$
The interest in doing this is as follows. If we drop the first $2$ from the first triangle, we get
$$
\begin{array}{cccccccccc}
 3 &  &  &  &  &  &  &  &  &  \\
 1 & 2 & 3 &  &  &  &  &  &  &  \\
 1 & 1 & 2 & 1 & 2 & 3 &  &  &  &  \\
 1 & 1 & 1 & 2 & 1 & 1 & 2 & 1 & 2 & 3 \\
\end{array}
$$
We can continue the construction of this irregular triangle by using the rules (Benoit Cloitre)
$$ 3 \to 123, \quad 2 \to 12,\quad 1 \to 1.$$
The corresponding  sequence for positive $4$ in $\tilde{s}_n$ is \seqnum{A212013}, which begins
$$1, 3, 4, 7, 9, 10, 14, 17, 19, 20, 25, 29, 32, 34, 35, 41, 46,\ldots.$$
To describe the general term of this sequence we use the function
$$trinv(n)=\lfloor \frac{1+\sqrt{8n+1}}{2} \rfloor = \lfloor \sqrt{2(n+1)}+\frac{1}{2} \rfloor.$$
Then the general term of this sequence is given by
$$\binom{trinv(n)}{3}+n+\left(n-\binom{trinv(n)}{2}\right)\left(trinv(n)(trinv(n)+3)-2n-2\right)/4.$$
This sequence is described in the OEIS as giving the total number of pairs of states of the first $n$ sub-shells of the nuclear shell model in which the sub-shells are ordered by energy level in increasing order.

Insight into the structure of this sequence can be obtained by representing it as a square array, read by anti-diagonals, as follows.
$$\left(
\begin{array}{ccccccc}
 1 & 4 & 10 & 20 & 35 & 56 & 84 \\
 3 & 9 & 19 & 34 & 55 & 83 & 119 \\
 7 & 17 & 32 & 53 & 81 & 117 & 162 \\
 14 & 29 & 50 & 78 & 114 & 159 & 214 \\
 25 & 46 & 74 & 110 & 155 & 210 & 276 \\
 41 & 69 & 105 & 150 & 205 & 271 & 349 \\
 63 & 99 & 144 & 199 & 265 & 343 & 434 \\
\end{array}
\right).$$
Transforming this by multiplying on the left by the inverse binomial matrix, and on the right by its transpose, leads to the simple encoding
$$\left(
\begin{array}{ccccccc}
 1 & 3 & 3 & 1 & 0 & 0 & 0 \\
 2 & 3 & 1 & 0 & 0 & 0 & 0 \\
 2 & 1 & 0 & 0 & 0 & 0 & 0 \\
 1 & 0 & 0 & 0 & 0 & 0 & 0 \\
 0 & 0 & 0 & 0 & 0 & 0 & 0 \\
 0 & 0 & 0 & 0 & 0 & 0 & 0 \\
 0 & 0 & 0 & 0 & 0 & 0 & 0 \\
\end{array}
\right).$$
This has generating function
$$(1 + 2x + 2x^2 + x^3) + y(3 + 3x + x^2) + y^2(3 + x) + y^3.$$
By undoing the binomial transforms in $x$ and $y$, we obtain the result that the generating function of the square array for the sequence \seqnum{A212013} is given by
$$\frac{1+(y^2-3y-1)x-(y^3-2y^2-y-1)x^2-yx^3}{(1-x)^4(1-y)^4}.$$
As a triangle read by row, \seqnum{A212013} will have then have generating function
$$\frac{1+(x^2y^2-3xy-1)x-(x^3y^3-2x^2y^2-xy-1)x^2-yx^4}{(1-x)^4(1-xy)^4}.$$
\section{Some further results}
The sequence
$$1,-1,-1,-2,-5,-14,-42,\ldots$$  has the Hankel transform
$$1,0,-1,-2,-3,-4,-5,-6,\ldots.$$
Taking this sequence modulo $2$, we obtain the sequence
$$1,0,1,0,1,0,1,0,1,0,1,0,\ldots.$$
The Hankel transform of the sequence $r_n^{(-)}$, which begins
$$1,-1, -1, 0, -1, 0, 0, 0, -1, 0, 0, 0,\ldots,$$ then ``inherits'' this property.
Thus the Hankel transform
$$1, -2, 3, 2, -3, 4, 3, 2, -3, 4, -5,\ldots$$ gives the parity sequence
$$1,0,1,0,1,0,1,0,1,0,1,0,\ldots.$$
Now consider the sequence
$$1,-2,-1,-2,-5,-14,-42,\ldots.$$
The Hankel transform of this sequence begins
$$1, -3, -28, -92, -213, -409, -698, -1098, -1627, -2303, \ldots,$$
which is the sequence $\frac{2+n-3n^2-6n^3}{2}.$ Taken modulo $2$, we obtain the sequence
$$1, 1, 0, 0, 1, 1, 0, 0, 1, 1, 0, 0, 1, 1, 0, \ldots.$$
Similarly, the sequence
$$1,-2, -1, 0, -1, 0, 0, 0, -1, 0, 0, 0,\ldots,$$ has a Hankel transform that begins
$$1, -5, 6, 2, -3, 7, 6, 2, -3, 7, -8, -4, -3, 7, 6, 2, \ldots.$$
Taking this modulo $2$, we again obtain
$$1, 1, 0, 0, 1, 1, 0, 0, 1, 1, 0, 0, 1, 1, 0, \ldots.$$

The difference between the Hankel transforms of
$$1,-2,-1,-2,-5,-14,-42,\ldots \quad\text{and }\quad 1,-1,-1,-2,-5,-14,-42,\ldots$$
is given by
$$0, -3, -27, -90, -210, -405, -693, -1092, -1620, -2295, -3135, \ldots.$$
This is $\frac{3(n+1)(1-2n)}{2}$, which is divisible by $3$.
Taking this modulo $2$, we obtain the sequence
$$0, 1, 1, 0, 0, 1, 1, 0, 0, 1, 1.$$

The difference between the Hankel transforms of
$$1,-2, -1, 0, -1, 0, 0, 0, -1, 0, 0, 0,\ldots, \quad\text{and}\quad 1,-1, -1, 0, -1, 0, 0, 0, -1, 0, 0, 0,\ldots$$ is given by
$$0, -3, 3, 0, 0, 3, 3, 0, 0, 3, -3, 0, 0, 3, 3, 0, 0, 3, -3, 0, 0,\ldots.$$
This is divisible by $3$, and modulo $2$ gives the sequence
$$0, 1, 1, 0, 0, 1, 1, 0, 0, 1, 1.$$

Taking the sequence that begins
$$1,-3, -1, 0, -1, 0, 0, 0, -1, 0, 0, 0,\ldots$$ now, we see that this has a Hankel transform that begins
$$1, -10, 11, 2, -3, 12, 11, 2, -3, 12, -13, -4, -3, 12, 11, 2, \ldots.$$
The difference between this and the sequence $1, -5, 6, 2, -3, 7, 6, 2, -3,\ldots$ is
$$0, -5, 5, 0, 0, 5, 5, 0, 0, 5, -5, 0, 0, 5, 5, 0, 0, 5, -5, 0, 0,\ldots.$$
In general, the modified Rueppel sequence $r_n^{(-)}(k)$ that begins
$$1,-k, -1, 0, -1, 0, 0, 0, -1, 0, 0, 0,\ldots$$ has a Hankel transform that begins
$$1, - k^2 - 1, k^2 + 2, 2, -3, k^2 + 3, k^2 + 2, 2, -3, k^2 + 3, - k^2 - 4, -4, -3, \ldots.$$
This is equal to the sequence that begins
$$0, - k^2, k^2, 0, 0, k^2, k^2, 0, 0, k^2, - k^2, 0, 0, k^2, k^2, 0, 0, k^2, - k^2, 0, 0,\ldots$$ added to the sequence (given by $k=0$)
$$1, -1, 2, 2, -3, 3, 2, 2, -3, 3, -4, -4, -3, 3, 2, 2, -3, 3, -4, -4, 5,\ldots.$$
This last sequence is the interpolation of the sequence beginning
$$-1, 2, 3, 2, 3, -4, 3, 2, 3, -4,\ldots$$ into the sequence that begins
$$1, 2, -3, 2, -3, -4, -3, 2, -3, -4,\ldots.$$
\begin{proposition}
The sequence
$$1, -1, 2, 2, -3, 3, 2, 2, -3, 3, -4, -4,\ldots$$ is the Hankel transform of the aerated $r_n^{(-)}$ sequence.
\end{proposition}
\begin{proof}
This follows since the sequence with $k$=0 coincides with the aeration of $r_n^{(-)}$. That is, we have
$$r_n^{(-)}(0)=r_{\frac{n}{2}}^{(-)} \frac{1+(-1)^n}{2}.$$
The interpolation result then follows from the general result on the Hankel transform of aerated sequences \cite{Wall}.
\end{proof}
Thus the Hankel transform of $r_n^{(-)}(k)$ is given by
$$ \{1,-1,2,2,-3,3,2,2,\ldots\}-k^2\{0,1,-1,0,0,-1,-1,0,0,-1,1,0,\ldots\}.$$
The second sequence $d_n$ that begins
$$0,1,-1,0,0,-1,-1,0,0,-1,1,0,\ldots$$ deserves some exploration. We conjecture the following.
We have $d_{4n}=d_{4n+3}=d_{8n+7}=0$ for all $n \ge 0$. To characterize $d_{4n+1}$, we form the sequence
$$1-\frac{d_{4n+1}+1}{2},$$ which begins
$$0, 1, 1, 1, 1, 0, 1, 1, 1, 0, 0, 0, 1, 0, 1, 1, 1, 0, 0, 0, 0, 1, 0, 0, 1, 0,\ldots.$$
We posit that this is \seqnum{A268411}, which gives the parity of the number of runs of $1$'s in binary representation of $n$. The sequence $d_{4n+2}$ appears more elusive.

A more complete characterization of the sequence $d_n$ appears to be as follows.
\begin{itemize}
\item The $0$'s are located at positions $n$ where $n$ is congruent to $0$ or $3 \bmod 4$ (\seqnum{A014601})
\item The values $\pm 1$ are located at positions $n$ where $n$ is congruent to $1$ or $2 \bmod 4$ (\seqnum{A042963})
\item The distribution of $1$'s is indexed (within \seqnum{A042963}) by the balanced evil numbers (\seqnum{A268412})
\item The distribution of the $-1$'s is indexed by the balanced odious numbers (\seqnum{A268415}).
\end{itemize}
The balanced evil numbers are those numbers which have an even number of runs of $1$'s in their binary expansion, while the balanced odious numbers have an odd number of runs of $1$'s in their binary expansion.

We can in fact display the sequence $d_{n+1}$ as a Hankel transform in a manner that explains its structure. Thus we consider the generating functions $\frac{1}{r(x)}$ and $\frac{x}{r(x)-1}=\frac{1}{r(x^2)}$.
The generating function $\frac{1}{r(x)}$ expands to give the sequence \seqnum{A104977} which begins
$$1, -1, 1, -2, 3, -4, 6, -10, 15, -22, 34, -52, 78, -118, 180, -274,\ldots.$$
We describe this sequence as ``Hankel rich'' in the sense that its Hankel transform, and that of successive shifts, are of significance. For the moment, we consider its own Hankel transform, which is given by
$$1, 0, -1, 0, 1, 0, -1, 0, 1, 0, -1, 0, 1, 0, -1, 0, 1, 0, -1, 0, 1,\ldots$$ with generating function
$$\frac{1}{1+x^2}.$$
Now the Hankel transform of the corresponding aerated sequence
$$1, 0, -1, 0, 1, 0, -2, 0, 3, 0, -4, 0, 6, 0, -10, 0, 15, 0, -22, 0, 34,0,\ldots$$  will then begin
$$1, -1, 0, 0, -1, -1, 0, 0, -1, 1, 0, 0, -1, -1, 0, 0, -1, 1, 0, 0, 1,\ldots.$$
This is the product of the Hankel transform of the expansion of $\frac{1}{r(x)}$, doubled, and the left shifted doubled Hankel transform of the expansion  of $(\frac{1}{r(x)}-1)/x$. This latter sequence begins
$$-1, 1, -2, 3, -4, 6, -10, 15, -22, 34, -52, 78, -118, 180, -274,\ldots,$$ with Hankel transform
which begins $$-1, 1, 1, -1, 1, 1, 1, -1, 1, -1, -1, -1, 1, 1, 1, -1, 1, -1, -1, 1,\ldots.$$
Thus the Hankel transform of the aerated sequence is the term-by-term product of
$$1, 1, 0, 0, -1, -1, 0, 0, 1, 1, 0, 0, -1, -1, 0, 0, 1, 1, 0, 0, -1, -1, 0,\ldots$$
with
$$1, -1, -1, 1, 1, 1, 1, -1, -1, 1, 1, 1, 1, 1, 1, -1, -1, 1, 1, -1, -1, -1, -1, \ldots.$$
In considering these generating functions, we are led to the following conjecture.
\begin{conjecture} The Hankel transform $h_n$ of the expansion of $\frac{x r(x)}{r(x)-1}$ satisfies
$$|h_n|=g_{n+1}.$$
\end{conjecture}
The sequence $h_n$ of this conjecture begins
$$1, -2, -1, 2, -3, -2, -1, 2, -3, 4, 3, 2, -3, -2, -1, 2, -3, 4, 3,\ldots.$$
We shall denote this sequence, pre-pended with a $0$, by $\tilde{g}_n$. We then have the following conjecture.
\begin{conjecture}
$$ \tilde{g}_n + \tilde{s}_n = \sigma_n.$$
\end{conjecture}
We can visualise this as follows. Consider the sequence $r_n(k)$  that is the expansion of $1- kx r(x)$. The Hankel transform of the sequence $r_n(k)$ begins
$$1, - k^2 - k, 2k^3 + k^2, k^4 + k^3, - 2k^5 - k^4, 3k^6 + k^5, 2k^7 + k^6, k^8 + k^7, - 2k^9 - k^8,\ldots.$$
This polynomial sequence has a coefficient array that begins
$$\left(
\begin{array}{cccccccccc}
 1 & 0 & 0 & 0 & 0 & 0 & 0 & 0 & 0 & 0 \\
 0 & -1 & -1 & 0 & 0 & 0 & 0 & 0 & 0 & 0 \\
 0 & 0 & 1 & 2 & 0 & 0 & 0 & 0 & 0 & 0 \\
 0 & 0 & 0 & 1 & 1 & 0 & 0 & 0 & 0 & 0 \\
 0 & 0 & 0 & 0 & -1 & -2 & 0 & 0 & 0 & 0 \\
 0 & 0 & 0 & 0 & 0 & 1 & 3 & 0 & 0 & 0 \\
 0 & 0 & 0 & 0 & 0 & 0 & 1 & 2 & 0 & 0 \\
 0 & 0 & 0 & 0 & 0 & 0 & 0 & 1 & 1 & 0 \\
 0 & 0 & 0 & 0 & 0 & 0 & 0 & 0 & -1 & -2 \\
\end{array}
\right).$$ The row sums (corresponding to $k=1$) are then given by $\tilde{s}_n$, while the two diagonals are, respectively, $\sigma_n$ and $\tilde{g}_n$.

For the Catalan based analogous sequence that begins $1,-r,-1,-2,-5,-14,-42,\ldots$ we have the following result. For $r=0$, the Hankel transform begins
$$1, 1, 0, -4, -13, -29, -54, -90, -139, -203, -284, -384, -505, -649, \ldots.$$
This is the sequence $\frac{6-n+3n^2-2n^3}{6}$. Denoting this sequence by $a_n$, the Hankel transform for arbitrary $r$ is then given by
$$a_n -4 r \binom{n+1}{3} - r^2 \left(\binom{n+1}{3}+\binom{n+2}{3}\right).$$

In like fashion, the Hankel transform of the sequence $r_n^{(+)}(r)$
$$1, r, 1, 0, 1, 0, 0, 0, 1, 0, 0, 0, 0,\ldots,$$ which begins
$$1, 1 - r^2, - r^2, 0, 1, r^2 + 1, - r^2, 0, 1,\ldots,$$
is composed of the sum of the Hankel transform
$$1, 1, 0, 0, 1, 1, 0, 0, 1, 1, 2, -2, 1, 1, 0, 0, 1,\ldots$$ of the sequence defined by $r=0$ and
$r^2$ times the sequence that begins
$$0, -1, -1, 0, 0, 1, -1, 0, 0, 1, 1, 0, 0, 1, -1, 0, 0, 1, 1, 0, 0,\ldots.$$

\section{Related explorations}
The sequence $\tilde{s}_n$ arose from taking the Hankel transform of a modified version of the Rueppel sequence. We explore other Hankel transforms briefly, and compare these with the Catalan case.
\begin{example}
We consider the Catalan variant that begins
$$1,r,2,5,14,42,\ldots.$$
The Hankel transform of this parameterized sequence is given by
$$h_n=-\frac{4n^6+24n^5+55n^4-1149n^2-114n-180}{180}+r \frac{n(n-1)(4n^4+28n^3+113n+54)}{90}$$
$$ \quad\quad\quad -r^n \frac{n+1}{12}\binom{2n+4}{5},$$
with generating function
$$\frac{1-5x+2x^2-21x^3+12x^4-6x^5+x^6}{(1-x)^7}+r \frac{4x^2(5+2x+x^2)}{(1-x)^7}$$
$$ \quad\quad\quad - r^2\frac{x(1+x)(1+6x+x^2)}{(1-x)^7}.$$
For $r=-2,\ldots,2$, we get the following values.
$$
\begin{array}{ccccccccccc}
 1 & -2 & -101 & -695 & -2849 & -8798 & -22595 & -50903 & -103931 & -196514  & \ldots \\
 1 & 1 & -39 & -295 & -1239 & -3863 & -9967 & -22511 & -46031 & -87119  & \ldots \\
 1 & 2 & -5 & -63 & -289 & -930 & -2435 & -5543 & -11387 & -21614 &  \ldots \\
 1 & 1 & 1 & 1 & 1 & 1 & 1 & 1 & 1 & 1 &  \ldots\\
 1 & -2 & -21 & -103 & -369 & -1070 & -2659 & -5879 & -11867 & -22274 &  \ldots \\
\end{array}
$$
For this parameterization, we can see the that case of the Catalan numbers ($r=1$) is special. Taken modulo $2$, we obtain
$$
\begin{array}{ccccccccccc}
 1 & 0 & 1 & 1 & 1 & 0 & 1 & 1 & 1 & 0 & \ldots \\
 1 & 1 & 1 & 1 & 1 & 1 & 1 & 1 & 1 & 1 & \ldots \\
 1 & 0 & 1 & 1 & 1 & 0 & 1 & 1 & 1 & 0 & \ldots \\
 1 & 1 & 1 & 1 & 1 & 1 & 1 & 1 & 1 & 1 & \ldots \\
 1 & 0 & 1 & 1 & 1 & 0 & 1 & 1 & 1 & 0 & \ldots \\
\end{array}
$$
We now look at the analogous sequences with respect to the Rueppel sequence. Thus we take the sequences of the form
$$1, r, 0, 1, 0, 0, 0, 1, 0, 0, 0,\ldots.$$
We find that the Hankel transform is given by
$$1, - r^2, -1, 1, 1, - r^2, -1, 1, 1, - r^2, -1, 1, 1, - r^2, -1, 1,\ldots,$$
or
$$ \overline{\{1,0,-1,1\}}-r^2 \overline{\{0,1,0,0\}}$$ with generating function
$$\frac{1-r^2x-x^2+x^3}{1-x^4}.$$
For $r=-2,\ldots,2$ we get the values
$$
\begin{array}{cccccccccccc}
 1 & -4 & -1 & 1 & 1 & -4 & -1 & 1 & 1 & -4 & -1 & \ldots \\
 1 & -1 & -1 & 1 & 1 & -1 & -1 & 1 & 1 & -1 & -1 & \ldots\\
 1 & 0 & -1 & 1 & 1 & 0 & -1 & 1 & 1 & 0 & -1 & \ldots\\
 1 & -1 & -1 & 1 & 1 & -1 & -1 & 1 & 1 & -1 & -1 & \ldots\\
 1 & -4 & -1 & 1 & 1 & -4 & -1 & 1 & 1 & -4 & -1 & \ldots \\
\end{array}
$$
The case $r=1$ is that of the Rueppel sequence, with Hankel transform $(-1)^{\binom{n+1}{2}}$. We note that the sequence for $r=-1$ also has this Hankel transform. The Jacobi parameters in the case $r=-1$ are, respectively,
$$-1, 0, 2, -2, 0, 2, 0, -2, 0, 0, 2,\ldots,$$ and
$$-1,-1,-1,-1,-1,-1,-1,-1,\ldots.$$ For the case of general $r$, these parameters are, respectively,
$$r,-r-\frac{1}{r^2}, \frac{1}{r^2}-r,r-1,r+1,\frac{1}{r^2}-r,-r-\frac{1}{r^2},r-1,r+1,-r-\frac{1}{r^2},\ldots,$$ and
$$-r^2,-\frac{1}{r^4},-r^2,-1,-r^2,-\frac{1}{r^4},-r^2,-1,-r^2,-\frac{1}{r^2},-r^2,\ldots.$$
Again, taken modulo $2$, we obtain
$$
\begin{array}{ccccccccccc}
 1 & 0 & 1 & 1 & 1 & 0 & 1 & 1 & 1 & 0 & \ldots \\
 1 & 1 & 1 & 1 & 1 & 1 & 1 & 1 & 1 & 1 & \ldots \\
 1 & 0 & 1 & 1 & 1 & 0 & 1 & 1 & 1 & 0 & \ldots \\
 1 & 1 & 1 & 1 & 1 & 1 & 1 & 1 & 1 & 1 & \ldots \\
 1 & 0 & 1 & 1 & 1 & 0 & 1 & 1 & 1 & 0 & \ldots \\
\end{array}
$$
\end{example}

\section{Some notable triangles and related sequences}
Many number triangles of combinatorial significance arise from the Catalan numbers. Thus the triangles
$(1, xc(x))$, $(c(x), xc(x))$, $(1, xc(x)^2)$, $(c(x), xc(x)^2)$ and many more are important. Here, we have used Riordan array notation \cite{book, SGWW}. For two suitable power series, $g(x)$ and $f(x)$, the triangle defined by $g(x)$ and $f(x)$ has its $(n,k)$-th element given by
$$ [x^n] g(x)f(x)^k,$$ where the functional $[x^n]$ extracts the coefficient of $x^n$.
We are thus motivated to look at such number triangles as those defined by $(1, r(x))$, $(r(x),xr(x))$ and others. Often, interesting results arise by looking at these triangles modulo $2$, where they will be the same as their Catalan counter-parts.

The triangle defined by $(1, xr(x))$ begins
$$\left(
\begin{array}{ccccccc}
 1 & 0 & 0 & 0 & 0 & 0 & 0 \\
 0 & 1 & 0 & 0 & 0 & 0 & 0 \\
 0 & 1 & 1 & 0 & 0 & 0 & 0 \\
 0 & 0 & 2 & 1 & 0 & 0 & 0 \\
 0 & 1 & 1 & 3 & 1 & 0 & 0 \\
 0 & 0 & 2 & 3 & 4 & 1 & 0 \\
 0 & 0 & 2 & 4 & 6 & 5 & 1 \\
\end{array}
\right).$$
Its row sums will have generating function $\frac{1}{1-xr(x)}$. This expands to give the sequence \seqnum{A023359}, which begins
$$1, 1, 2, 3, 6, 10, 18, 31, 56, 98, 174, \ldots.$$
This counts the number of compositions of $n$ into powers of $2$. Taken modulo $2$, this sequence returns the Rueppel sequence. This corresponds to the defining identity for the Catalan numbers
$$c(x)=\frac{1}{1-x c(x)}.$$ The Hankel transform of \seqnum{A023359} is of interest. It begins
$$1, 1, 1, -1, -1, -1, 1, -1, -1, -1, -1, 1, -1, \ldots,$$ so it is a signed version of that of the Catalan numbers, as is expected. The pattern of signs is not clear. The Jacobi parameters are
$$1,0,0,0,\ldots$$ for the $\alpha$ coefficients and the $\beta$ coefficients begin
$$1, 1, -1, -1, 1, -1, 1, -1, 1, 1, -1, 1, -1, -1, 1, -1,\ldots.$$
Forming the sequence $\frac{\beta_n+1}{2}$ we obtain
$$1, 1, 0, 0, 1, 0, 1, 0, 1, 1, 0, 1, 0, 0, 1, 0, 1, 1, 0, \ldots,$$
which we conjecture coincides with \seqnum{A090678}, or the sequence \seqnum{ A088567} of non-squashing partitions of $n$ into distinct parts taken modulo $2$. The single $1$ in the $\alpha$ sequence indicates that this sequence is the INVERT transform of the aerated sequence given by
$$\cfrac{1}{1-\cfrac{\beta_1 x^2}{1-\cfrac{\beta_2 x^2}{1-\cdots}}}.$$ We can obtain the generating function of this sequence as follows.
$$\frac{\frac{1}{1- x r(x)}}{1+\frac{x}{1- x r(x)}}=\frac{1}{1+x(1-r(x))}.$$
The generating function $\frac{1}{1-x(r(x)-1)}=\frac{1}{1-xr(x^2)}$ then expands to give the sequence that begins
$$1, 0, 1, 0, 2, 0, 3, 0, 6, 0, 10, \ldots.$$
As it has the same $\beta$ sequence, its Hankel transform is equal to that of its un-aerated version.

The triangle defined by $(1,xr(x))$ taken modulo $2$ begins
$$\left(
\begin{array}{ccccccc}
 1 & 0 & 0 & 0 & 0 & 0 & 0 \\
 0 & 1 & 0 & 0 & 0 & 0 & 0 \\
 0 & 1 & 1 & 0 & 0 & 0 & 0 \\
 0 & 0 & 0 & 1 & 0 & 0 & 0 \\
 0 & 1 & 1 & 1 & 1 & 0 & 0 \\
 0 & 0 & 0 & 1 & 0 & 1 & 0 \\
 0 & 0 & 0 & 0 & 0 & 1 & 1 \\
\end{array}
\right).$$
Its row sums are then given by \seqnum{A080100}, which begins
$$1, 1, 2, 1, 4, 2, 2, 1, 8, 4, 4, 2, 4, 2, 2, 1,\ldots.$$ The $n$-th element of this sequence is equal to
$2^{\text{number of $0$'s in the binary representation of $n$}}$.

The inverse of the triangle defined by $(1, xr(x))$ begins
$$\left(
\begin{array}{ccccccc}
 1 & 0 & 0 & 0 & 0 & 0 & 0 \\
 0 & 1 & 0 & 0 & 0 & 0 & 0 \\
 0 & -1 & 1 & 0 & 0 & 0 & 0 \\
 0 & 2 & -2 & 1 & 0 & 0 & 0 \\
 0 & -6 & 5 & -3 & 1 & 0 & 0 \\
 0 & 20 & -16 & 9 & -4 & 1 & 0 \\
 0 & -70 & 56 & -31 & 14 & -5 & 1 \\
\end{array}
\right).$$
The sequence
$$0, 1, -1, 2, -6, 20, -70, 256, -970, 3772, -14960,\ldots$$ is the reversion of the sequence given by $xr(x)$. This is \seqnum{A092413}. Modulo $2$ this is the sequence $0,1,1,0,0,\ldots$ with generating function $x(1+x)$. This is to be expected as the corresponding Catalan case (the reversion of $xc(x)$) is given by $x(1-x)$. The row sums, which begin
$$1, 1, 0, 1, -3, 10, -35, 129, -492, 1921, -7641,\ldots$$ have residues modulo $2$ given by
$$1, 1, 0, 1, 1, 0, 1, 1, 0, 1, 1,\ldots.$$ Again, in the corresponding Catalan case we have row sums given by
$$1, 1, 0, -1, -1, 0, 1, 1, 0, -1, -1,\ldots$$ with the same residues.

Taken modulo $2$, the inverse matrix begins
$$\left(
\begin{array}{ccccccc}
 1 & 0 & 0 & 0 & 0 & 0 & 0 \\
 0 & 1 & 0 & 0 & 0 & 0 & 0 \\
 0 & 1 & 1 & 0 & 0 & 0 & 0 \\
 0 & 0 & 0 & 1 & 0 & 0 & 0 \\
 0 & 0 & 1 & 1 & 1 & 0 & 0 \\
 0 & 0 & 0 & 1 & 0 & 1 & 0 \\
 0 & 0 & 0 & 1 & 0 & 1 & 1 \\
\end{array}
\right).$$
Its row sums then begin
$$1, 1, 2, 1, 3, 2, 3, 1, 4, 3, 5,\ldots.$$ This is the Stern's diatomic series \seqnum{A002487}.

The triangle defined by $(r(x), xr(x))$ begins
$$\left(
\begin{array}{ccccccc}
 1 & 0 & 0 & 0 & 0 & 0 & 0 \\
 1 & 1 & 0 & 0 & 0 & 0 & 0 \\
 0 & 2 & 1 & 0 & 0 & 0 & 0 \\
 1 & 1 & 3 & 1 & 0 & 0 & 0 \\
 0 & 2 & 3 & 4 & 1 & 0 & 0 \\
 0 & 2 & 4 & 6 & 5 & 1 & 0 \\
 0 & 0 & 6 & 8 & 10 & 6 & 1 \\
\end{array}
\right).$$
Its row sums have generating function given by $\frac{r(x)}{1-xr(x)}$, which is the INVERT transform of $r(x)$. This sequence begins
$$1, 2, 3, 6, 10, 18, 31, 56, 98, 174, 306,\ldots.$$ As the INVERT transform of the Rueppel numbers, it has the same Hankel transform as that of the Rueppel sequence.

The row sums of this triangle modulo $2$ begin $1, 2, 1, 4, 2, 2, 1, 8, 4, 4,\ldots$, the once-shifted \seqnum{A080100}. Similarly, the inverse triangle, taken modulo $2$, will have row sums equal to the Stern sequence starting $1, 2, 1, 3, 2, 3, 1, \ldots$.

We next look at triangles defined by the reciprocal generating function $\frac{1}{r(x)}$. The triangle defined by $\left(1, \frac{x}{r(x)}\right)$ begins
$$\left(
\begin{array}{ccccccc}
 1 & 0 & 0 & 0 & 0 & 0 & 0 \\
 0 & 1 & 0 & 0 & 0 & 0 & 0 \\
 0 & -1 & 1 & 0 & 0 & 0 & 0 \\
 0 & 1 & -2 & 1 & 0 & 0 & 0 \\
 0 & -2 & 3 & -3 & 1 & 0 & 0 \\
 0 & 3 & -6 & 6 & -4 & 1 & 0 \\
 0 & -4 & 11 & -13 & 10 & -5 & 1 \\
\end{array}
\right).$$
It row sums will have generating function given by
$$\frac{1}{1-\frac{x}{r(x)}}=\frac{r(x)}{r(x)-x}.$$
The expansion of this generating function begins
$$1, 1, 0, 0, -1, 0, 0, 1, -1, 0, -1, 2, 0, 1, -3, 1,\ldots.$$
The Hankel transform of this sequence begins
$$1, -1, 1, 1, -1, 1, 1, 1, -1, 1, -1, -1, -1, 1, 1,\ldots.$$
We have the following conjecture.
\begin{conjecture}
The Hankel transform $h_n$ of the expansion of  $\frac{r(x)}{r(x)-x}$, where $r(x)$ is the generating function of the Rueppel numbers, satisfies
$$h_n=\sigma_n.$$
\end{conjecture}
Further insight into this conjecture can be obtained by parameterization. Thus we consider the expansion of $$\frac{r(x)}{f(x)-kx} = \left(1, \frac{x}{f(x)}\right)\cdot \frac{1}{1-kx}.$$ This expansion begins
$$1, k, -k + k^2, k - 2 k^2 + k^3, -2 k + 3 k^2 - 3 k^3 + k^4,
3 k - 6 k^2 + 6 k^3 - 4 k^4 + k^5,\ldots.$$ The Hankel transform of this polynomial sequence begins
$$1, -k, k^2, k^3, -k^4, k^5, k^6, k^7, -k^8, k^9, -k^10,\ldots.$$ In matrix form this is
$$\left(
\begin{array}{ccccccccccc}
 1 & 0 & 0 & 0 & 0 & 0 & 0 & 0 & 0 & 0 & 0 \\
 0 & -1 & 0 & 0 & 0 & 0 & 0 & 0 & 0 & 0 & 0 \\
 0 & 0 & 1 & 0 & 0 & 0 & 0 & 0 & 0 & 0 & 0 \\
 0 & 0 & 0 & 1 & 0 & 0 & 0 & 0 & 0 & 0 & 0 \\
 0 & 0 & 0 & 0 & -1 & 0 & 0 & 0 & 0 & 0 & 0 \\
 0 & 0 & 0 & 0 & 0 & 1 & 0 & 0 & 0 & 0 & 0 \\
 0 & 0 & 0 & 0 & 0 & 0 & 1 & 0 & 0 & 0 & 0 \\
 0 & 0 & 0 & 0 & 0 & 0 & 0 & 1 & 0 & 0 & 0 \\
 0 & 0 & 0 & 0 & 0 & 0 & 0 & 0 & -1 & 0 & 0 \\
 0 & 0 & 0 & 0 & 0 & 0 & 0 & 0 & 0 & 1 & 0 \\
 0 & 0 & 0 & 0 & 0 & 0 & 0 & 0 & 0 & 0 & -1 \\
\end{array}
\right)\cdot \left(\begin{array} {c} 1 \\ k \\ k^2\\ k^3 \\ k^4 \\k^5 \\ k^6 \\ k^7 \\ k^8 \\ k^9 \\ k^10 \end{array} \right).$$
The above matrix is then an involution (its square is the identity, or equivalently, it is self-inverse) with diagonal entries given by $\sigma_n$.

We note that in the analogous Catalan case, the corresponding sign sequence turns out to be $(-1)^n$ and the corresponding involution is then $(1, -x)$.

We note that the Hankel transform of the expansion of the reciprocal $\frac{f(x)-x}{f(x)}=1-\frac{x}{f(x)}$ begins
$$1, 0, 0, -1, -1, 0, 0, -1, -1, 0, 0, 1, -1, 0, 0, -1, -1, 0, 0,\ldots.$$

Taken modulo $2$, the triangle defined by $\left(1, \frac{x}{f(x)}\right)$ begins
$$\left(
\begin{array}{cccccccc}
 1 & 0 & 0 & 0 & 0 & 0 & 0 & 0 \\
 0 & 1 & 0 & 0 & 0 & 0 & 0 & 0 \\
 0 & 1 & 1 & 0 & 0 & 0 & 0 & 0 \\
 0 & 1 & 0 & 1 & 0 & 0 & 0 & 0 \\
 0 & 0 & 1 & 1 & 1 & 0 & 0 & 0 \\
 0 & 1 & 0 & 0 & 0 & 1 & 0 & 0 \\
 0 & 0 & 1 & 1 & 0 & 1 & 1 & 0 \\
 0 & 0 & 0 & 1 & 0 & 1 & 0 & 1 \\
\end{array}
\right).$$
Its row sums appear to be the sequence \seqnum{A214126}, which begins
$$1, 1, 2, 2, 3, 2, 4, 3, 5, 2, 5, 4, 6,\ldots.$$ This sequence $a_n$ is defined by
$$a_{2n}=a_{n-1}+a_n, a_{2n+1}=a_{n+1},$$ with $a_0=1, a_1=1$.

The inverse of this triangle begins
$$\left(
\begin{array}{ccccccc}
 1 & 0 & 0 & 0 & 0 & 0 & 0 \\
 0 & 1 & 0 & 0 & 0 & 0 & 0 \\
 0 & 1 & 1 & 0 & 0 & 0 & 0 \\
 0 & 1 & 2 & 1 & 0 & 0 & 0 \\
 0 & 2 & 3 & 3 & 1 & 0 & 0 \\
 0 & 5 & 6 & 6 & 4 & 1 & 0 \\
 0 & 11 & 15 & 13 & 10 & 5 & 1 \\
\end{array}
\right).$$
The sequence $0,1,1,1,2,5,11,\ldots$ is the expansion of the reversion of $\frac{x}{r(x)}$. It is documented in \seqnum{A134527}. David Scambler has noted there that it counts the  number of Dyck $n$-paths with all ascent lengths being $1$ less than a power of $2$. The residues of the sequence $1,1,1,2,5,11,\ldots$ modulo $2$ coincide with those of the ternary numbers $\frac{1}{n+1}\binom{3n}{n+1}$, given by \seqnum{A085357}. This corresponds to the fact that the reversion of $\frac{x}{c(x)}$ is the generating function of the ternary numbers (with a $0$ pre-pended).

We now look at the number triangle defined by $(1, x(1-xr(x))$. We recall that $1-xr(x)=1 - x - x^2 - x^4 - x^8 - x^{16} - \cdots$ is the generating function of $r_n^{(-1)}$. This triangle begins
$$\left(
\begin{array}{ccccccc}
 1 & 0 & 0 & 0 & 0 & 0 & 0 \\
 0 & 1 & 0 & 0 & 0 & 0 & 0 \\
 0 & -1 & 1 & 0 & 0 & 0 & 0 \\
 0 & -1 & -2 & 1 & 0 & 0 & 0 \\
 0 & 0 & -1 & -3 & 1 & 0 & 0 \\
 0 & -1 & 2 & 0 & -4 & 1 & 0 \\
 0 & 0 & -1 & 5 & 2 & -5 & 1 \\
\end{array}
\right).$$
Its row sums have generating function $\frac{1}{1-x-x^2r(x)}$, and begin
$$ 1, 1, 0, -2, -3, -2, 2, 7, 9, 2, -13, -26, -20, 13, 59,\ldots.$$
Their Hankel transform is given by $(-1)^n$, and the Jacobi parameters $\alpha_n$ and $\beta_n$ begin
$$1, 1, -2, 0, 0, 2, 0, -2, 0, 2,\ldots,$$ and
$$-1,-1,-1,-1,-1,\ldots,$$ respectively. We conjecture that the $\alpha$ sequence is a signed version of \seqnum{A110036}, where for $n >1$, we have $|\alpha_n|=2\cdot$\seqnum{A088567}$(n) \bmod 2$, where as we have seen
\seqnum{A088567} counts non-squashing partitions of $n$ into distinct parts.

The shifted sequence that begins
$$1, 0, -2, -3, -2, 2, 7, 9, 2, -13, -26, -20, 13, 59,\ldots,$$ which gives the row sum of the triangle defined by $(1-xr(x), x(1-xr(x)))$ has Hankel transform equal to $\tilde{s}_n$, namely
$$1, -2, 3, 2, -3, 4, 3, 2, -3, 4,\ldots.$$ This is because the generating function of the row sums of $(1-r(x), x(1-rx))$ is given by
$$\frac{1-xr(x)}{1-x(1-xr(x))},$$ which is the INVERT transform of $1-xr(x)$.

The row sums of this triangle taken modulo $2$ are \seqnum{A214126} beginning
$$1, 1, 2, 2, 3, 2, 4, 3, 5, 2, 5, 4, 6, 3, 7, 5, 8,\ldots.$$
The inverse of the triangle $(1, x(1-xr(x))$ begins
$$\left(
\begin{array}{ccccccc}
 1 & 0 & 0 & 0 & 0 & 0 & 0 \\
 0 & 1 & 0 & 0 & 0 & 0 & 0 \\
 0 & 1 & 1 & 0 & 0 & 0 & 0 \\
 0 & 3 & 2 & 1 & 0 & 0 & 0 \\
 0 & 10 & 7 & 3 & 1 & 0 & 0 \\
 0 & 39 & 26 & 12 & 4 & 1 & 0 \\
 0 & 161 & 107 & 49 & 18 & 5 & 1 \\
\end{array}
\right).$$
The sequence beginning
$$0, 1, 1, 3, 10, 39, 161, 698, 3126, 14361, 67287, 320319, 1544894,\ldots$$
is the expansion of the reversion of $x(1-xr(x))$. Taken modulo $2$, the sequence $1,1,3,10,\ldots$ gives \seqnum{A085357}, the residues modulo $2$ of the ternary numbers. Once again, we note that the reversion of $x(1-x c(x))=\frac{x}{c(x)}$ is the generating function of the ternary numbers (with $0$ prepended).

Taking the inverse of the triangle $(1, x(1-xr(x)))$ and then reducing this modulo $2$, gives us a row sum sequence that begins
$$1, 1, 2, 2, 3, 2, 5, 3, 4, 2, 6, 5, 8, 3, 7, 4, 5, 2, 7, 6, 10,\ldots.$$

We finish this section by looking at the number triangle defined by the Riordan array
$$(1-xr(x), xr(x)).$$
This triangle begins
$$\left(
\begin{array}{ccccccc}
 1 & 0 & 0 & 0 & 0 & 0 & 0 \\
 -1 & 1 & 0 & 0 & 0 & 0 & 0 \\
 -1 & 0 & 1 & 0 & 0 & 0 & 0 \\
 0 & -2 & 1 & 1 & 0 & 0 & 0 \\
 -1 & 0 & -2 & 2 & 1 & 0 & 0 \\
 0 & -2 & -1 & -1 & 3 & 1 & 0 \\
 0 & -2 & -2 & -2 & 1 & 4 & 1 \\
\end{array}
\right).$$
The sequence $0,1,0,-2,0,-2,-2,\ldots$ with generating function $xr(x)(1-xr(x))$ is a signed version of \seqnum{A151758}. The row sums of this triangle are $1,0,0,0,\ldots$ since the generating function of the row sums is given by $\frac{1-xr(x)}{1-xr(x)}=1$. The diagonal sums of this matrix begin
$$1, -1, 0, 0, -2, 1, -3, -1, -3, -5, -4, -11, -10, -20, -25,\ldots,$$ with Hankel transform that begins
$$1, -1, 2, 3, -5, 2, 3, 7, -10, 3, 1, -2, -3, -1, 4, 13, -17,\ldots.$$
Taken modulo $2$, this gives
$$1, 1, 0, 1, 1, 0, 1, 1, 0, 1, 1, 0, 1, 1, 0, 1, 1,0,\ldots.$$
The equivalent diagonal sum sequence for the Catalan analog begins
$$1, -1, 0, -2, -4, -13, -39, -125, -409, -1371, -4678, -16203,\ldots$$ with Hankel transform
$$1, -1, 0, 1, -1, 0, 1, -1, 0, 1, -1, 0, 1, -1, 0, 1, -1, 0,\ldots.$$
The inverse of this triangle begins
$$\left(
\begin{array}{ccccccc}
 1 & 0 & 0 & 0 & 0 & 0 & 0 \\
 1 & 1 & 0 & 0 & 0 & 0 & 0 \\
 1 & 0 & 1 & 0 & 0 & 0 & 0 \\
 1 & 2 & -1 & 1 & 0 & 0 & 0 \\
 1 & -4 & 4 & -2 & 1 & 0 & 0 \\
 1 & 16 & -12 & 7 & -3 & 1 & 0 \\
 1 & -54 & 44 & -24 & 11 & -4 & 1 \\
\end{array}
\right).$$
Modulo $2$, we obtain the triangle that begins
$$\left(
\begin{array}{ccccccccccc}
 1 & 0 & 0 & 0 & 0 & 0 & 0 & 0 & 0 & 0 & 0 \\
 1 & 1 & 0 & 0 & 0 & 0 & 0 & 0 & 0 & 0 & 0 \\
 1 & 0 & 1 & 0 & 0 & 0 & 0 & 0 & 0 & 0 & 0 \\
 1 & 0 & 1 & 1 & 0 & 0 & 0 & 0 & 0 & 0 & 0 \\
 1 & 0 & 0 & 0 & 1 & 0 & 0 & 0 & 0 & 0 & 0 \\
 1 & 0 & 0 & 1 & 1 & 1 & 0 & 0 & 0 & 0 & 0 \\
 1 & 0 & 0 & 0 & 1 & 0 & 1 & 0 & 0 & 0 & 0 \\
 1 & 0 & 0 & 0 & 1 & 0 & 1 & 1 & 0 & 0 & 0 \\
 1 & 0 & 0 & 0 & 0 & 0 & 0 & 0 & 1 & 0 & 0 \\
 1 & 0 & 0 & 0 & 0 & 1 & 0 & 1 & 1 & 1 & 0 \\
 1 & 0 & 0 & 0 & 0 & 0 & 1 & 0 & 1 & 0 & 1 \\
\end{array}
\right).$$
The row sums of this triangle begin
$$1, 2, 2, 3, 2, 4, 3, 4, 2, 5, 4, 6, 3, 6, 4, 5,\ldots.$$
We recall that Stern's diatomic sequence $st_n$ \seqnum{A002487} begins
$$0, 1, 1, 2, 1, 3, 2, 3, 1, 4, 3, 5, 2, 5, 3, 4, 1,\ldots.$$
The general term $st_n$ of this sequence can be defined as
$$st_n=\sum_{k=0}^{n-1} \binom{k}{n-k-1} \bmod 2.$$
We define a sequence $a_n$ as follows. First, $a_0=1$. Then
$$a_{2n}=a_n, \quad a_{2n+1}=a_{n-1}+st_{n+1}.$$
\begin{conjecture}
The row sums of the matrix
$$(1-xr(x), xr(x))^{-1} \bmod 2$$ are given by $a_n$ defined above.
\end{conjecture}

\section{The sequence $\sigma_n$}
We recall that the sequence $\sigma_n$ has been defined as
$$\sigma_n= \frac{|\tilde{s}_n|}{\tilde{s}_n},$$ where the sequence $s_n$ is the Hankel transform of the expansion of $1-xr(x)$.

The Thue-Morse sequence $tm_n$ is \seqnum{A010060} which begins
$$	0, 1, 1, 0, 1, 0, 0, 1, 1, 0, 0, 1, 0, 1, 1, 0, 1, 0,\ldots,$$ is a classical sequence, whose doubled sequence \seqnum{A095190} begins
$$0, 0, 1, 1, 1, 1, 0, 0, 1, 1, 0, 0, 0, 0, 1, 1, 1, 1, 0, 0, 0, 0, 1, 1, 0, \ldots.$$
The Thue-Morse sequence can be calculated as
$$tm_n=1-\text{Mod}\left(\text{Mod}\left(\sum_{k=0}^n \text{Mod}\left(\binom{n}{k},2\right),3\right),2\right).$$
Thus the doubled sequence has general term
$$tm_{\lfloor \frac{n}{2} \rfloor}=1-\text{Mod}\left(\text{Mod}\left(\sum_{k=0}^{\lfloor \frac{n}{2} \rfloor} \text{Mod}\left(\binom{\lfloor \frac{n}{2} \rfloor}{k},2\right),3\right),2\right).$$
The Golay-Rudin-Shapiro sequence \seqnum{A020985} begins
$$	1, 1, 1, -1, 1, 1, -1, 1, 1, 1, 1, -1, -1, -1, 1, -1, 1, 1, 1,\ldots.$$ This sequence has links to paper-folding \cite{Mendes}. It can be defined as the sequence $a_n$ such that $a_0=1$, and
$$a_{2n}=a_n,\quad  a_{2n+1} = (-1)^n a_n .$$
Alternatively, its $n$-th term is given by
$$(-1)^{\text{number of occurrences of $11$ in the binary expansion of $n$}}.$$
We calculate $2 a_n-\sigma_n.$ This begins
$$0, 0, 1, 1, -1, 1, 0, 0, -1, 1, 0, 0, 0, 0, 1, 1, -1, 1, 0, 0, 0, 0, -1, -1, 0, 0,\ldots.$$
This leads to the following conjecture.
\begin{conjecture}
Let $a_n$ be the Golay-Rudin-Shapiro sequence, and let $tm_n$ be the Thue-Morse sequence. Then we have
$$ | 2 a_n-\sigma_n|=tm_{\lfloor \frac{n}{2} \rfloor}.$$
\end{conjecture}
The sequence \seqnum{A268411}, the parity of the number of runs of $1$'s in the binary representation of $n$, begins $$0, 1, 1, 1, 1, 0, 1, 1, 1, 0, 0, 0, 1, 0, 1, 1, 1, 0, 0, 0, 0, 1,\ldots.$$
It can be defined as follow, where we denote this sequence by $b_n$.
We have $b_0=0, b_{2n}=b_n$, and for odd $n$, $b_{2n+1}=b_n$, while for even $n$, $b_{2n+1}=1-b_n$.
We then have the following conjecture.
\begin{conjecture} Let $b_n$ be the sequence \seqnum{A268411} that gives the parity of the number of runs of $1$'s in the binary representation of $n$. Then we have
$$|\sigma_n - (-1)^{\binom{n}{2}}|=2b_n.$$
\end{conjecture}
We have that
the sequence $\frac{\sigma_n-(-1)^{\binom{n}{2}}}{2}$ begins
$$0, -1, 1, 1, -1, 0, 1, 1, -1, 0, 0, 0, -1, 0, 1, 1, -1, 0, 0, 0, 0,
-1, 0,\ldots,$$ and we thus conjecture that this is a signed version of \seqnum{A268411}.
In fact, we can refine the above conjecture to the following.
\begin{conjecture}
Let $R_n$ \cite{Sh} be the sequence \seqnum{A268411} that gives the parity of the number of runs of $1$'s in the binary representation of $n$. Then we have
$$\sigma_n=\frac{(-1)^{\binom{n}{2}}}{1-2 R_n}.$$
\end{conjecture}
In fact, we have another conjectured formula for $\sigma_n$ in terms of $R_n$. 
\begin{conjecture}
We have 
$$\sigma_n=\left(1-2R_{\frac{n}{2}}\right)(-1)^{\frac{n}{2}}\frac{1+(-1)^n}{2}-\left(1-2R_{\frac{n-1}{2}}\right)\frac{1-(-1)^n}{2}.$$ \end{conjecture}
We also conjecture that $\sigma_{2n}$ and $-\sigma_{2n+1}$ share the same Hankel transform, that begins 
$$1, -2, 0, 8, -48, -32, 320, -768, -4352, -2560, -3072,\ldots.$$
We let $\sigma(x)=\sum_{n=0}^{\infty} \sigma_n x^n$ be the generating function of the sequence $\sigma_n$. It is of interest to investigate the array $(\sigma(x), x \sigma(x))$. This array begins
$$\left(
\begin{array}{ccccccccc}
 1 & 0 & 0 & 0 & 0 & 0 & 0 & 0 & 0 \\
 -1 & 1 & 0 & 0 & 0 & 0 & 0 & 0 & 0 \\
 1 & -2 & 1 & 0 & 0 & 0 & 0 & 0 & 0 \\
 1 & 3 & -3 & 1 & 0 & 0 & 0 & 0 & 0 \\
 -1 & 0 & 6 & -4 & 1 & 0 & 0 & 0 & 0 \\
 1 & -3 & -4 & 10 & -5 & 1 & 0 & 0 & 0 \\
 1 & 6 & -3 & -12 & 15 & -6 & 1 & 0 & 0 \\
 1 & -1 & 15 & 3 & -25 & 21 & -7 & 1 & 0 \\
 -1 & 0 & -14 & 24 & 20 & -44 & 28 & -8 & 1 \\
\end{array}
\right).$$
Modulo $2$, we retrieve the Sierpinski triangle
$$\left(
\begin{array}{ccccccccccccccccccccc}
 1 & \text{} & \text{} & \text{} & \text{} & \text{} & \text{} & \text{} & \text{} & \text{} & \text{} & \text{} & \text{} & \text{} & \text{} & \text{} & \text{} & \text{} &
   \text{} & \text{} & \text{} \\
 1 & 1 & \text{} & \text{} & \text{} & \text{} & \text{} & \text{} & \text{} & \text{} & \text{} & \text{} & \text{} & \text{} & \text{} & \text{} & \text{} & \text{} & \text{}
   & \text{} & \text{} \\
 1 & \text{} & 1 & \text{} & \text{} & \text{} & \text{} & \text{} & \text{} & \text{} & \text{} & \text{} & \text{} & \text{} & \text{} & \text{} & \text{} & \text{} & \text{}
   & \text{} & \text{} \\
 1 & 1 & 1 & 1 & \text{} & \text{} & \text{} & \text{} & \text{} & \text{} & \text{} & \text{} & \text{} & \text{} & \text{} & \text{} & \text{} & \text{} & \text{} & \text{} &
   \text{} \\
 1 & \text{} & \text{} & \text{} & 1 & \text{} & \text{} & \text{} & \text{} & \text{} & \text{} & \text{} & \text{} & \text{} & \text{} & \text{} & \text{} & \text{} & \text{}
   & \text{} & \text{} \\
 1 & 1 & \text{} & \text{} & 1 & 1 & \text{} & \text{} & \text{} & \text{} & \text{} & \text{} & \text{} & \text{} & \text{} & \text{} & \text{} & \text{} & \text{} & \text{} &
   \text{} \\
 1 & \text{} & 1 & \text{} & 1 & \text{} & 1 & \text{} & \text{} & \text{} & \text{} & \text{} & \text{} & \text{} & \text{} & \text{} & \text{} & \text{} & \text{} & \text{} &
   \text{} \\
 1 & 1 & 1 & 1 & 1 & 1 & 1 & 1 & \text{} & \text{} & \text{} & \text{} & \text{} & \text{} & \text{} & \text{} & \text{} & \text{} & \text{} & \text{} & \text{} \\
 1 & \text{} & \text{} & \text{} & \text{} & \text{} & \text{} & \text{} & 1 & \text{} & \text{} & \text{} & \text{} & \text{} & \text{} & \text{} & \text{} & \text{} & \text{}
   & \text{} & \text{} \\
 1 & 1 & \text{} & \text{} & \text{} & \text{} & \text{} & \text{} & 1 & 1 & \text{} & \text{} & \text{} & \text{} & \text{} & \text{} & \text{} & \text{} & \text{} & \text{} &
   \text{} \\
 1 & \text{} & 1 & \text{} & \text{} & \text{} & \text{} & \text{} & 1 & \text{} & 1 & \text{} & \text{} & \text{} & \text{} & \text{} & \text{} & \text{} & \text{} & \text{} &
   \text{} \\
 1 & 1 & 1 & 1 & \text{} & \text{} & \text{} & \text{} & 1 & 1 & 1 & 1 & \text{} & \text{} & \text{} & \text{} & \text{} & \text{} & \text{} & \text{} & \text{} \\
 1 & \text{} & \text{} & \text{} & 1 & \text{} & \text{} & \text{} & 1 & \text{} & \text{} & \text{} & 1 & \text{} & \text{} & \text{} & \text{} & \text{} & \text{} & \text{} &
   \text{} \\
 1 & 1 & \text{} & \text{} & 1 & 1 & \text{} & \text{} & 1 & 1 & \text{} & \text{} & 1 & 1 & \text{} & \text{} & \text{} & \text{} & \text{} & \text{} & \text{} \\
 1 & \text{} & 1 & \text{} & 1 & \text{} & 1 & \text{} & 1 & \text{} & 1 & \text{} & 1 & \text{} & 1 & \text{} & \text{} & \text{} & \text{} & \text{} & \text{} \\
 1 & 1 & 1 & 1 & 1 & 1 & 1 & 1 & 1 & 1 & 1 & 1 & 1 & 1 & 1 & 1 & \text{} & \text{} & \text{} & \text{} & \text{} \\
 1 & \text{} & \text{} & \text{} & \text{} & \text{} & \text{} & \text{} & \text{} & \text{} & \text{} & \text{} & \text{} & \text{} & \text{} & \text{} & 1 & \text{} & \text{}
   & \text{} & \text{} \\
 1 & 1 & \text{} & \text{} & \text{} & \text{} & \text{} & \text{} & \text{} & \text{} & \text{} & \text{} & \text{} & \text{} & \text{} & \text{} & 1 & 1 & \text{} & \text{} &
   \text{} \\
 1 & \text{} & 1 & \text{} & \text{} & \text{} & \text{} & \text{} & \text{} & \text{} & \text{} & \text{} & \text{} & \text{} & \text{} & \text{} & 1 & \text{} & 1 & \text{} &
   \text{} \\
 1 & 1 & 1 & 1 & \text{} & \text{} & \text{} & \text{} & \text{} & \text{} & \text{} & \text{} & \text{} & \text{} & \text{} & \text{} & 1 & 1 & 1 & 1 & \text{} \\
 1 & \text{} & \text{} & \text{} & 1 & \text{} & \text{} & \text{} & \text{} & \text{} & \text{} & \text{} & \text{} & \text{} & \text{} & \text{} & 1 & \text{} & \text{} &
   \text{} & 1 \\
\end{array}
\right)$$ with row sums given by Gould's sequence. Moreover, vertical and horizontal half triangles, taken modulo $2$, have row sums given by the Rueppel sequence and the augmented Rueppel sequence, respectively. In this respect, the array $(\sigma(x), x \sigma(x))$ is cognate with Pascal's triangle.

Similarly, the array $\left(1, \frac{x}{\sigma(x)}\right)$ when taken modulo $2$ has row sums given by Stern's diatomic sequence, and its vertical and horizontal halves, again taken modulo $2$, have row sums  respectively given by \seqnum{A007306} (denominators of Farey tree fractions) and Gould's sequence. Thus the array $\left(1, \frac{x}{\sigma(x)}\right)$ is cognate with the array $\left(\binom{k}{n-k}\right)$ which corresponds to $(1, x(1-x))$.

Taking the array $\left(1, \frac{x}{\sigma(x)}\right)^{-1}$ modulo $2$ gives an array that begins
$$\left(
\begin{array}{ccccccccccccccccccccc}
 1 & \text{} & \text{} & \text{} & \text{} & \text{} & \text{} & \text{} & \text{} & \text{} & \text{} & \text{} & \text{} & \text{} & \text{} & \text{} & \text{} & \text{} &
   \text{} & \text{} & \text{} \\
 \text{} & 1 & \text{} & \text{} & \text{} & \text{} & \text{} & \text{} & \text{} & \text{} & \text{} & \text{} & \text{} & \text{} & \text{} & \text{} & \text{} & \text{} &
   \text{} & \text{} & \text{} \\
 \text{} & 1 & 1 & \text{} & \text{} & \text{} & \text{} & \text{} & \text{} & \text{} & \text{} & \text{} & \text{} & \text{} & \text{} & \text{} & \text{} & \text{} & \text{}
   & \text{} & \text{} \\
 \text{} & \text{} & \text{} & 1 & \text{} & \text{} & \text{} & \text{} & \text{} & \text{} & \text{} & \text{} & \text{} & \text{} & \text{} & \text{} & \text{} & \text{} &
   \text{} & \text{} & \text{} \\
 \text{} & 1 & 1 & 1 & 1 & \text{} & \text{} & \text{} & \text{} & \text{} & \text{} & \text{} & \text{} & \text{} & \text{} & \text{} & \text{} & \text{} & \text{} & \text{} &
   \text{} \\
 \text{} & \text{} & \text{} & 1 & \text{} & 1 & \text{} & \text{} & \text{} & \text{} & \text{} & \text{} & \text{} & \text{} & \text{} & \text{} & \text{} & \text{} & \text{}
   & \text{} & \text{} \\
 \text{} & \text{} & \text{} & \text{} & \text{} & 1 & 1 & \text{} & \text{} & \text{} & \text{} & \text{} & \text{} & \text{} & \text{} & \text{} & \text{} & \text{} & \text{}
   & \text{} & \text{} \\
 \text{} & \text{} & \text{} & \text{} & \text{} & \text{} & \text{} & 1 & \text{} & \text{} & \text{} & \text{} & \text{} & \text{} & \text{} & \text{} & \text{} & \text{} &
   \text{} & \text{} & \text{} \\
 \text{} & 1 & 1 & 1 & 1 & 1 & 1 & 1 & 1 & \text{} & \text{} & \text{} & \text{} & \text{} & \text{} & \text{} & \text{} & \text{} & \text{} & \text{} & \text{} \\
 \text{} & \text{} & \text{} & 1 & \text{} & 1 & \text{} & 1 & \text{} & 1 & \text{} & \text{} & \text{} & \text{} & \text{} & \text{} & \text{} & \text{} & \text{} & \text{} &
   \text{} \\
 \text{} & \text{} & \text{} & \text{} & \text{} & 1 & 1 & \text{} & \text{} & 1 & 1 & \text{} & \text{} & \text{} & \text{} & \text{} & \text{} & \text{} & \text{} & \text{} &
   \text{} \\
 \text{} & \text{} & \text{} & \text{} & \text{} & \text{} & \text{} & 1 & \text{} & \text{} & \text{} & 1 & \text{} & \text{} & \text{} & \text{} & \text{} & \text{} & \text{}
   & \text{} & \text{} \\
 \text{} & \text{} & \text{} & \text{} & \text{} & \text{} & \text{} & \text{} & \text{} & 1 & 1 & 1 & 1 & \text{} & \text{} & \text{} & \text{} & \text{} & \text{} & \text{} &
   \text{} \\
 \text{} & \text{} & \text{} & \text{} & \text{} & \text{} & \text{} & \text{} & \text{} & \text{} & \text{} & 1 & \text{} & 1 & \text{} & \text{} & \text{} & \text{} & \text{}
   & \text{} & \text{} \\
 \text{} & \text{} & \text{} & \text{} & \text{} & \text{} & \text{} & \text{} & \text{} & \text{} & \text{} & \text{} & \text{} & 1 & 1 & \text{} & \text{} & \text{} & \text{}
   & \text{} & \text{} \\
 \text{} & \text{} & \text{} & \text{} & \text{} & \text{} & \text{} & \text{} & \text{} & \text{} & \text{} & \text{} & \text{} & \text{} & \text{} & 1 & \text{} & \text{} &
   \text{} & \text{} & \text{} \\
 \text{} & 1 & 1 & 1 & 1 & 1 & 1 & 1 & 1 & 1 & 1 & 1 & 1 & 1 & 1 & 1 & 1 & \text{} & \text{} & \text{} & \text{} \\
 \text{} & \text{} & \text{} & 1 & \text{} & 1 & \text{} & 1 & \text{} & 1 & \text{} & 1 & \text{} & 1 & \text{} & 1 & \text{} & 1 & \text{} & \text{} & \text{} \\
 \text{} & \text{} & \text{} & \text{} & \text{} & 1 & 1 & \text{} & \text{} & 1 & 1 & \text{} & \text{} & 1 & 1 & \text{} & \text{} & 1 & 1 & \text{} & \text{} \\
 \text{} & \text{} & \text{} & \text{} & \text{} & \text{} & \text{} & 1 & \text{} & \text{} & \text{} & 1 & \text{} & \text{} & \text{} & 1 & \text{} & \text{} & \text{} & 1 &
   \text{} \\
 \text{} & \text{} & \text{} & \text{} & \text{} & \text{} & \text{} & \text{} & \text{} & 1 & 1 & 1 & 1 & \text{} & \text{} & \text{} & \text{} & 1 & 1 & 1 & 1 \\
\end{array}
\right).$$
The row sums of this matrix begin
$$1, 1, 2, 1, 4, 2, 2, 1, 8, 4, 4, 2, 4, 2, 2, 1, 16, 8, 8, 4, 8,\ldots,$$ which coincides with
\seqnum{A080100}, or $2^{\text{number of $0$'s in the binary representation of $n$}}$. Modulo $2$ this gives the Rueppel sequence.

\end{document}